\documentclass[11pt]{article}
\usepackage{authblk}
\usepackage{amsmath, amssymb, amsthm}
\usepackage{enumitem}
\usepackage[margin=1in]{geometry}
\usepackage{algorithmicx,algpseudocode}
\usepackage{graphicx,tikz,subfig}
\usetikzlibrary{calc,fit, positioning}

\newtheorem{theorem}{Theorem}
\newtheorem{lemma}[theorem]{Lemma}
\newtheorem{claim}{Claim}

\newtheorem{corollary}[theorem]{Corollary}

\theoremstyle{definition}
\newtheorem{definition}[theorem]{Definition}

\theoremstyle{definition}
\newtheorem{case}{Case}

\newcommand{\subcase}[1]{\vspace{4pt}{\bf Subcase {#1}}}

\newcommand{\cB}{\mathcal{B}}
\newcommand{\cC}{\mathcal{C}}
\newcommand{\cA}{\mathcal{A}}

\def\rad{1.3cm}
\def\inRad{.8cm}
\def\nodeSize{4mm}

\newcommand\drawNode[5]{
\draw (#1,#2) circle(\rad);
\foreach \x in {#3}
\node[draw,fill,black,circle,minimum size=\nodeSize, inner sep=0pt,xshift=#1cm,yshift=#2cm] (B#5-\x) at (\x:\inRad) {}; 
\foreach \x in {#4}
\node[draw,fill,red,circle,minimum size=\nodeSize, inner sep=0pt,xshift=#1cm,yshift=#2cm] (R#5-\x) at (\x:\inRad) {}; 
}

\newcommand\hideNode[5]{\draw[draw=none] (#1,#2) circle(\rad);}

\begin{document}
\title{Finding Independent Transversals Efficiently}
\author{Alessandra Graf} 
\author{Penny Haxell\thanks{Partially supported by NSERC.}}
\affil{Department of Combinatorics and Optimization, University of Waterloo, Waterloo, ON, Canada\\ (e-mail: \texttt{agraf@uwaterloo.ca}, \texttt{pehaxell@uwaterloo.ca})}

\maketitle

\begin{abstract}
We give an efficient algorithm that, given a graph $G$ and a partition $V_1,\ldots,V_m$ of its vertex set, finds either an {\it independent transversal} (an independent set $\{v_1,\ldots,v_m\}$ in $G$ such that $v_i\in V_i$ for each $i$), or a subset $\mathcal B$ of vertex classes such that the subgraph of $G$ induced by $\bigcup\mathcal B$ has a small dominating set. A non-algorithmic proof of this result has been known for a number of years and has been applied to solve many other problems. Thus we are able to give algorithmic versions of many of these applications, a few of which we describe explicitly here. 
\end{abstract}

\section{Introduction}

Let $G$ be a graph whose vertex set is partitioned into classes $V_1,\ldots,V_m$. An {\it independent transversal} (IT) of $G$ with respect to the given vertex partition is an independent set $\{v_1,\ldots,v_m\}$ in $G$ such that $v_i\in V_i$ for each $i$. This is a very general notion, and many combinatorial problems can be formulated by asking if a given graph with a given vertex partition has an IT. Indeed the SAT problem can be formulated in these terms (see e.g. \cite{Haxell2011}), and so we cannot expect to find an efficient characterisation of those $G$ for which an IT exists. However, there are now various known results giving sufficient conditions for the existence of an IT. One of the most easily stated and most frequently applied is the following result from~\cite{Haxell1995, Haxell2001}.

\begin{theorem}\label{maxdeg}
Let $G$ be a graph with maximum degree $\Delta$. Then for any vertex
partition $(V_1,\ldots,V_m)$ of $G$ where $|V_i|\geq2\Delta$ for each $i$, there
exists an independent transversal of $G$. 
\end{theorem}

Theorem~\ref{maxdeg} answered the question of how large the vertex
classes need to be, in terms of the maximum degree, to guarantee the
existence of an IT in $G$. This question was first introduced and
studied in 1975 by Bollob\'as, Erd\H os and Szemer\'edi in
\cite{Bollobas1975}, and further progress was contributed over the
years by many authors. In particular, linear upper bounds in terms of
$\Delta$ were given by Alon~\cite{Alon1988} (in an early application
of the Lov\'asz Local Lemma~\cite{Lovasz1975})  and independently
Fellows~\cite{Fellows1990}, and a later application of the Local Lemma
gave that class size $2e\Delta$ is sufficient (see e.g. Alon and
Spencer~\cite{Alon2008}). Further refining this approach, Bissacot, Fern\'andez, Procacci and Scoppola~\cite{Bissacot2011} improved this to $4\Delta$.  Work on lower bounds 
 included results of Jin~\cite{Jin1992}, Yuster~\cite{ Yuster1997}, and
Alon~\cite{ Alon2003}, and in 2006
Szab\'o and Tardos~\cite{Szabo2006} gave constructions for every
$\Delta$ in which $|V_i|=2\Delta-1$ for each $i$ but there is no IT.
Therefore Theorem~\ref{maxdeg} is best possible for every value of $\Delta$. A
more precise version of Theorem~\ref{maxdeg} (involving also the
number $m$ of vertex classes) is given in \cite{Haxell2004}. 

Theorem~\ref{maxdeg} is an immediate consequence of a more general
statement described in terms of domination
(stated explicitly in~\cite{Haxell2001}, although it follows easily
from the argument in~\cite{Haxell1995}). We
say that a subset $D\subseteq V(G)$ {\it dominates} a subgraph
$W$ of $G$ if for all $w\in V(W)$, there exists $uw\in E(G)$ for
some $u\in D$. (This definition of domination is quite often referred
to as {\it strong domination} or {\it total domination}, but since it is the only notion of
domination that we will refer to in this paper, we will use the
simpler term.) 
For a vertex partition $(V_1,\ldots,V_m)$ of $G$ and a subset $\cB$ of
$\{V_1,\ldots,V_m\}$, we write $G_{\cB}$ for the subgraph
$G[\bigcup_{V_i\in\cB}V_i]-\{uv\in E(G):u,v\in V_i\hbox{ for some }V_i\in\cB\}$ obtained by removing any edges inside vertex classes from the subgraph of $G$ induced by the union of the classes in $\cB$.  

\begin{theorem}\label{maxtrans}
Let $G$ be a graph with a vertex partition $(V_1,\ldots,V_m)$. Suppose
that, for each $\cB\subseteq\{V_1,\ldots,V_m\}$, the subgraph
$G_{\cB}$ is not dominated in $G_{\cB}$ by any set of size at most
$2(|\cB|-1)$. Then $G$ has an IT.
\end{theorem}
To see that Theorem~\ref{maxtrans} implies Theorem~\ref{maxdeg},
simply note that if the union of $|\cB|$ vertex classes in $G$ contains
a total of $2\Delta|\cB|$ vertices then $G_{\cB}$ cannot be dominated by
$2|\cB|-2$ vertices of degree at most $\Delta$. 

In fact the proof of Theorem~\ref{maxtrans} shows that if $G$ does not
have an IT then there exists $\cB\subseteq\{V_1,\ldots,V_m\}$ such that
$G_{\cB}$ is dominated by the vertex set of a {\it constellation} for
$\cB$.
\begin{definition}\label{constell}
Let $\cB$ be a set of vertex classes in a vertex-partitioned graph
$G$. A {\it constellation} for $\cB$ is an induced subgraph $K$ of
$G_{\cB}$, whose components are stars with at least two vertices, each with
a {\it centre} and a nonempty set of {\it leaves} distinct from its
centre. The set of all leaves of $K$ forms an IT of $|\cB|-1$
vertex classes of $\cB$.
\end{definition}
Note that if $K$ is a constellation for $\cB$ then
$|V(K)|\leq2(|\cB|-1)$. Figure~\ref{const} shows an example of a constellation. 

\begin{figure}[!ht]
\centering
\scalebox{.36}{
\begin{tikzpicture}
\drawNode{1}{0}{-180,0}{}{a}
\drawNode{-4}{-2}{-30,-150}{90}{b}
\drawNode{6}{-2}{-30,-150}{90}{c}
\drawNode{-6}{-5}{-90}{90}{d}
\drawNode{-2}{-5}{}{90}{e}
\drawNode{2}{-5}{-90}{90}{f}
\drawNode{6}{-5}{-90}{90}{g}
\drawNode{10}{-5}{-90}{90}{h}
\drawNode{-7}{-8}{}{90}{i}
\drawNode{-4}{-8}{}{90}{j}
\drawNode{-1}{-8}{}{90}{k}
\drawNode{2}{-8}{}{90}{l}
\drawNode{5}{-8}{}{90}{m}
\drawNode{8}{-8}{}{90}{n}
\drawNode{11}{-8}{}{90}{o}

\draw[line width=2pt] (Ba--180)--(Rb-90);
\draw[line width=2pt] (Ba-0)--(Rc-90);
\draw[line width=2pt] (Bb--150)--(Rd-90);
\draw[line width=2pt] (Bb--30)--(Re-90);
\draw[line width=2pt] (Bc--150)--(Rf-90);
\draw[line width=2pt] (Bc--150)--(Rg-90);
\draw[line width=2pt] (Bc--30)--(Rh-90);
\draw[line width=2pt] (Bd--90)--(Ri-90);
\draw[line width=2pt] (Bd--90)--(Rj-90);
\draw[line width=2pt] (Bd--90)--(Rk-90);
\draw[line width=2pt] (Bf--90)--(Rl-90);
\draw[line width=2pt] (Bg--90)--(Rm-90);
\draw[line width=2pt] (Bh--90)--(Rn-90);
\draw[line width=2pt] (Bh--90)--(Ro-90);

\node [left = 7 mm of Ba--180,font=\fontsize{30}{0}](label){$x_1$};
\node [right = 6 mm of Ba-0,font=\fontsize{30}{0}](label){$x_3$};
\node [left = 4 mm of Bb--150,font=\fontsize{30}{0}](label){$x_2$};
\node [right = 4 mm of Bb--30,font=\fontsize{30}{0}](label){$x_4$};
\node [left = 5 mm of Bc--150,font=\fontsize{30}{0}](label){$x_7$};
\node [right = 4 mm of Bc--30,font=\fontsize{30}{0}](label){$x_5$};
\node [right = 10 mm of Bd--90,font=\fontsize{30}{0}](label){$x_6$};
\node [left = 9 mm of Bf--90,font=\fontsize{30}{0}](label){$x_8$};
\node [left = 9 mm of Bg--90,font=\fontsize{30}{0}](label){$x_9$};
\node [left = 9 mm of Bh--90,font=\fontsize{30}{0}](label){$x_{10}$};

\draw (12,0) circle(\rad);
\draw (14,-5) circle(\rad);
\draw (16,-1) circle(\rad);
\draw (18,-4) circle(\rad);
\draw (17,-7.5) circle(\rad);

\draw[dashed,rounded corners=5pt, inner sep=4pt] 
(-9,-9.5) to (-8,-3) to (-4,1.5) to (7,1.5) to (11.5, -3) to (13,-9.5) to (-9,-9.5);
\end{tikzpicture}}
\caption{A constellation $K$ for the set $\cB$ of classes enclosed by
  the dotted border. Each circle represents a vertex class. The centres of the stars appear in black
 (and have labels) and the leaves appear in red.} 
\label{const}
\end{figure}
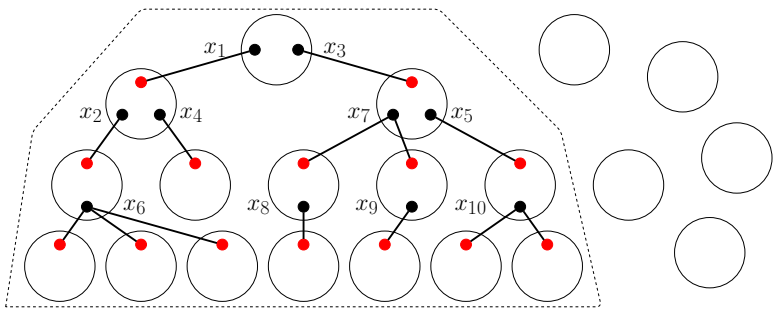

Theorem~\ref{maxtrans} (often in the form of Theorem~\ref{maxdeg}) has
been applied to obtain many results in various fields, including graph
theory (e.g. list colouring~\cite{Haxell2001}, strong colouring~\cite{Haxell2004, Aharoni2007}, delay edge colouring~\cite{Alon2007},
circular colouring~\cite{Kaiser2004, Kaiser2007}, various graph partitioning and special independent
set problems~\cite{Alon+2003,Haxell2003,King2011,Christofides2013}), hypergraphs (e.g. hypergraph matching~\cite{Haxell1995,Krivelevich1997,Annamalai2016,AnnamalaiThesis}), group theory
(e.g. generators in linear groups~\cite{Britnell2008}), and theoretical computer science
(e.g. job scheduling and other resource allocation
problems~\cite{Asadpour2008, Asadpour2012}). Unfortunately, the proofs of Theorems~\ref{maxdeg}
and~\ref{maxtrans} are not algorithmic. For certain applications it
is enough to know that class sizes of $c\Delta$ in
Theorem~\ref{maxdeg} guarantee an IT for some constant $c$, so for 
these one could obtain algorithmic versions using (for
example) the algorithmic Lov\'asz Local Lemma (\cite{Beck1991, Moser2009,
  Moser2010}). However, for many other applications, having the best
possible value of the constant $c$ is important.
This therefore raises the question of
how much the hypotheses in Theorems~\ref{maxtrans} and~\ref{maxdeg}
need to be strengthened in order to guarantee 
that an IT can be found efficiently, thereby giving
algorithmic proofs of these applications with the constants being as
close as possible to their optimal values. 

Most of the known results on this question have focused on Theorem~\ref{maxdeg}, and have been obtained as applications of algorithmic versions of the Lov\'asz Local Lemma or its lopsided variant. These include the original algorithm of Beck~\cite{Beck1991} and its improvements (see e.g. \cite{Alon1991, Molloy1998, Czumaj2000, Srinivasan2008, Bissacot2011}) as well as the resampling algorithm of Moser and Tardos~\cite{Moser2009} and its improvements (see e.g. \cite{Kolipaka2011, Pegden2014, Kolipaka2012, Harris2014, Achlioptas2016, Harris2013}. In particular, using the Moser-Tardos approach (and based on~\cite{Bissacot2011, Pegden2014}), Harris and Srinivasan~\cite{Harris2017} gave a randomized algorithm that finds an IT in expected time $O(m\Delta)$ in graphs with class size $4\Delta$. The current best result for polynomial expected time is due to Harris~\cite{Harris2016} who improved the bound on the class size to $4\Delta-1$. Deterministic algorithms based on derandomizing the Moser-Tardos algorithm have also been studied, but they require the class size be $C\Delta$ for some large constant $C$ in order to find an IT efficiently~\cite{Fischer2017,Harris2018}. Some of these deterministic algorithms are known to be parallelizable~\cite{Chandrasekaran2013, Harris2018}.

In this paper, we address the algorithmic IT question for a large class of
graphs, without using the Lov\'asz Local Lemma or any of its variants. A graph $G$ with vertex
partition $(V_1,\dots,V_m)$ is said to be {\it $r$-claw-free with
  respect to} $(V_1,\dots,V_m)$ if no vertex of $G$ has $r$
independent neighbours in distinct vertex classes. 
Our main theorem is as follows.   
\begin{theorem}\label{main} 
There exists an
algorithm FindITorBD that takes as input $r\in\mathbb{N}$ and
$\epsilon>0$, and a graph $G$ with vertex 
partition $(V_1,\dots,V_m)$ such that $G$ is $r$-claw-free with respect to
$(V_1,\dots,V_m)$, and finds either:  
\begin{enumerate}
\item an independent transversal in $G$, or
\item a set $\mathcal B$ of vertex classes and a set $D$ of vertices
  of $G$ such that $D$ dominates $G_{\cB}$ in $G$ and
  $|D|<(2+\epsilon)(|{\mathcal B}|-1)$. Moreover $D$ contains $V(K)$
  for a constellation $K$ for some $\cB_0\supseteq\cB$, where
  $|D\setminus V(K)|<\epsilon(|\cB|-1)$.
\end{enumerate}
For fixed $r$ and $\epsilon$, the runtime is polynomial in $|V(G)|$.
\end{theorem}

Note in particular that any graph with maximum degree $\Delta$ is
$(\Delta+1)$-claw-free with respect to any partition. Thus taking
$r=\Delta+1$ and $\epsilon=1/\Delta$ gives the following algorithmic version of
Theorem~\ref{maxdeg}.

\begin{corollary}\label{maxdegalg}
Let $\Delta\in\mathbb N$ be given. Then there exists an algorithm
that takes as input any graph $G$ with maximum degree $\Delta$ and
vertex partition $(V_1,\dots,V_m)$ such that ${|V_i|\ge2\Delta+1}$ for
each $i$ and finds, in time polynomial in $|V(G)|$, an independent
transversal in $G$. 
\end{corollary}
Therefore only a slight strengthening of the hypotheses is required to
make these results algorithmic. 

As shown in the proof of Theorem~\ref{main} in Section~\ref{sigs},
the runtime of the algorithm FindITorBD is 
$O(|V(G)|^g)$ where $g$ is a function of $r$ and $\epsilon$. Similarly for Corollary~\ref{maxdegalg} the degree $g$ depends on $\Delta$. We remark
that, for simplicity, in this paper we make no attempt 
to optimize the running time of our algorithms in terms of these
parameters.

The proof of Theorem~\ref{main} explicitly describes the algorithm
FindITorBD. It uses ideas from the original (non-algorithmic) proof
of Theorem~\ref{maxtrans} (see \cite{Haxell1995, Haxell2003}), and
modifications of several key notions (including that of ``lazy updates'') 
introduced by Annamalai in~\cite{Annamalai2016,Annamalai2017}, who
gave an algorithmic version of the
specific case of matchings in bipartite hypergraphs.
This appears as Theorem~\ref{r-partite} in
Section~\ref{apps}, and is relevant
to other well-studied problems such as the restricted max-min fair
allocation problem (also known as the Santa Claus problem), see~\cite{Asadpour2008, Asadpour2012}. Theorem~\ref{main} is a broad generalisation of
Theorem~\ref{r-partite} which, because of the large number of
applications of Theorems~\ref{maxdeg} and~\ref{maxtrans} over the years, has
algorithmic consequences for many results in a wide variety of settings.
In addition to describing the case of bipartite hypergraph matching, in
Section~\ref{apps} we outline algorithmic versions of a few more selected
applications of
Theorem~\ref{maxtrans} that follow from our results. Here we have
chosen to discuss circular chromatic index (Kaiser, Kr\'al and \v
 Skrekovski~\cite{Kaiser2004}), strong colouring (Aharoni, Berger and
 Ziv~\cite{Aharoni2007}), and hitting sets for maximum cliques
 (King~\cite{King2011}), but there are many other examples, some of
 which are described in detail in~\cite{thesis}. In each case, the
 algorithmic version is only slightly weaker than the original result
 due to the error $\epsilon$ introduced in Theorem~\ref{main}. In fact
 for some applications (for example the results on circular chromatic
 index and hitting sets for  maximum cliques) no weakening at all is needed.

This paper is organised as follows. In Section~\ref{over} we give an
overview of the proof of our main result, by first outlining the proof of
Theorem~\ref{maxtrans} (which gives an exponential algorithm) and then
sketching how we modify it to make the algorithm efficient. Our
algorithms are described in detail in
Section~\ref{algs}, after definitions and other preliminary material
in Section~\ref{prelims}. The main components of the algorithm are
analysed in Section~\ref{analysis} and the running time in
Section~\ref{sigs}. Section~\ref{apps} is devoted to applications of
our results, and Section~\ref{concrem} contains concluding remarks and
open questions. 

\section{Setup and Overview}\label{over}
Throughout this paper we will work with the following notation and
assumptions. Let $r$ and $\epsilon$ be fixed, let $G$ be a graph and
let $(V_1,\dots,V_m)$ be a vertex partition of 
$G$ such that $G$ is $r$-claw-free with respect to
$(V_1,\dots,V_m)$. By deleting the edges between vertices in the same
vertex class $V_i$ and considering the resulting graph $G'$, we may
assume without loss of generality that each vertex class $V_i$ is an
independent set of vertices. This is because a set $M$ is an IT
  of $G'$ if and only if $M$ is an IT of $G$. Since the case $m=1$ is
trivial we may assume from now on that $m\ge2$.

Our algorithms will seek to construct an IT of
$G$ step by step, by augmenting a previously constructed 
{\it partial independent transversal} (PIT) of $G$ with respect to the
given vertex partition. A PIT is simply an independent set $M$ in
$G$ (of size at most $m$) such that no two vertices of $M$ are in
the same vertex class.

Note that any isolated vertex can be added to any PIT that does not
contain a vertex in its vertex class. Thus, we may remove the vertex
classes from $V_1,\dots, V_m$ that contain at least one isolated
vertex and consider the induced subgraph of the remaining vertex
classes as $G$ under the same partition of these vertices. We will
therefore assume from now on that $G$ does not contain an isolated
vertex. In particular, we may also assume that $r\ge2$. 

We denote the vertex class that contains
the vertex $v\in V(G)$ by $A(v)$ and the set of vertex classes
containing $W\subseteq V(G)$ by ${A(W)=\{A(v):v\in W\}}$. We write
$N(v)$ for the neighbourhood in $G$ of $v$, and
$N_W(v)$ for $N(v)\cap W$. We denote $|N_W(v)|$ by $d_W(v)$.

To give an overview of the proof of Theorem~\ref{main}, we first
sketch the proof of Theorem~\ref{maxtrans}. The proof does give a
procedure for constructing an IT, but (as we will note after the
sketch) the number of steps could be as large as $(r-1)^m$.

\begin{proof}[Sketch of the proof of Theorem~\ref{maxtrans}:] 
Let $M$
  be a PIT and $A$ be a vertex class such that $A\cap M=\emptyset$. We
  aim to alter $M$ until it can be augmented by a vertex of $A$.

We build a ``tree-like structure'' $T$ (which we describe as a vertex
set inducing a forest of stars) as
follows. Choose $x_1\in A$ and set $T=\{x_1\}$.
If $d_M(x_1)=0$ then {\it improve} $M$ by adding $x_1$ to $M$, and stop.
Otherwise add $N_M(x_1)$ to $T$.

In the general $i$-th step: it can easily be shown that
$|T|\leq2(|A(T)|-1)$. Thus by assumption the subgraph $G_{A(T)}$ of
$G$ induced by $\bigcup_{V_i\in A(T)}V_i$ is not dominated by
$T$. Therefore there exists a vertex of 
$G_{A(T)}$ that is not adjacent to 
any vertex in $T$. Choose such a vertex $x_i$ arbitrarily.

If $d_M(x_i)=0$ then {\it improve} $M$ by adding $x_i$ to
$M$ and removing the $M$-vertex $y$ in $A(x_i)$ (if it exists). This forms
a new PIT $M$, and is an
improvement in the following sense: it
reduces $d_M(x_j)$ where $y$ was added to $T$ because it was in
$N_M(x_j)$ in an earlier step. Truncate $T$ to $\{x_1\}\cup
N_M(x_1)\cup\dots\cup\{x_j\}\cup N_M(x_j)$.

Otherwise add $x_i$ and $N_M(x_i)$  to $T$. Thus $|A(T)|$ increases by
$d_M(x_i)>0$ and $|T|$ increases by $d_M(x_i)+1$, thus maintaining
$|T|\leq2(|A(T)|-1)$. See Figure~\ref{const} for $T$ (the set of
vertices shown) and $A(T)$ (the set of classes enclosed by the
  dotted border). Note that $T$ is a constellation for $A(T)$, whose
  centres are the $x_i$ and whose leaves are the $N_M(x_i)$. 

At each step we either grow $T$ OR reduce $d_M(x_j)$ for some
$j$, UNTIL the current $M$ can be extended to include a vertex of $A$. 
Thus progress can be measured by a {\it signature vector}
$$(d_M(x_1),\ldots,d_M(x_t),\infty)$$
that has length at most $|M|+1\leq m$ since each $N_M(x_i)$ is a nonempty subset
of $M$ and all such sets are mutually disjoint. 
Each step reduces the lexicographic order of the
signature vector. Thus the process terminates, and we succeed in
extending $M$ to a larger PIT and eventually to an IT.
\end{proof}

The drawback of the above procedure is that the number of signature
vectors (and hence the number of steps) could potentially be as large
as $(r-1)^m$, where $G$ is $r$-claw-free. To make this approach into an
efficient algorithm, we make three main modifications. Here the idea
of ``lazy updates'' from~\cite{Annamalai2016, 
  Annamalai2017} is used, which essentially  amounts to performing
updates in ``clusters'' (large subsets of vertices) rather than at the level
of individual 
vertices (that change the quantities $d_M(x_i)$ only one at a
time). These modifications are as follows.

\begin{enumerate}
\item Maintain layers: at each growth step, instead of choosing $x_i$
  arbitrarily, choose it to be a vertex 
  in a class at smallest possible ``distance'' from the root class $A$,
  similar to a breadth-first
  search. Vertices $x_i$ added into classes at the same distance from
  $A$ are in the same {\it layer} (see Figure~\ref{treealg}). 
\item Update in ``clusters'': instead of updating $M$ when a single
  $x_i$ satisfies $d_M(x_i)=0$, update only when at least a
  {\it positive proportion} 
$\mu$ of an entire layer satisfies $d_M(x)=0$. Discard later layers.
\item Rebuild layers in ``clusters'': after an update, add new
  vertices $x_i$ to a layer of $T$ only if doing so would add a $\mu$ 
proportion of that layer. Then discard later layers.
\end{enumerate}
The parameter $\mu$ is a fixed positive constant, chosen to be small
enough with respect to the parameters $\epsilon$ and $1/r$. The extra
$\epsilon$ factor in (2) 
of Theorem~\ref{main} is enough to guarantee that the same proof idea
as for Theorem~\ref{maxtrans} finds an IT in $G$, UNLESS (as in that
proof) at some point in the construction of $T$, 
(a subset $\cB$ that is almost all of) $A(T)$ is dominated by $T$
(plus a certain very
small set of additional vertices, necessary to deal with the error
introduced by the lazy updates in Modifications 2 and 3). This dominating set
will have total size less than 
$(2+\epsilon)(|\cB|-1)$, resulting in Output (2) in
Theorem~\ref{main}. Thus for the rest of this section we will assume
that this never occurs, and in particular that a positive
proportion (depending on $\epsilon$) of the vertices of $G_{A(T)}$ are
not dominated by $T$.   
   
A consequence of {\it maintaining layers} is that the
vertices in $G_{A(T)}$ that do not have a neighbour in $T$ tend to be
``pushed'' towards the   
bottom layer. This results in the set of vertices of $T$ in the bottom
layer having size a positive proportion $\rho$ of $|T|$, where again
$\rho$ depends only on $\epsilon$ and $r$ (see Lemma~\ref{allbigx}). This
implies that the total number of layers is always logarithmic in
$m$ (Lemma~\ref{fewlayers}), since with each new layer the total size of $T$
increases by a fixed factor larger than one.

{\it Updating in clusters} and {\it rebuilding layers in clusters}
allow a different signature 
vector, that measures sizes of layers rather
than degrees of individual vertices $d_M(x_i)$. It has just two
entries per 
layer: the first is essentially -$\lfloor \log x\rfloor$ where $x$ is
the number of vertices $x_i$ in that layer, 
and the second is essentially $\lfloor \log y\rfloor$ where $y$ is the
total size of their neighbourhoods in 
$M$. Updating $M$ in a cluster (Modification 2)  decreases the
value of $y$ for a layer by a positive proportion. Rebuilding a layer
in a cluster (Modification 3) increases the value of $x$ for a
layer by a positive proportion. Hence (with suitably chosen bases for
the logarithms) these updates always {\it decrease} the relevant entry
by an integer amount. Therefore, as in the proof of
Theorem~\ref{maxtrans}, each update decreases the signature vector
lexicographically (Lemma~\ref{lexreduce}). 

Since the length of the signature vector is proportional to the number
of layers, as noted above this is logarithmic in $m$. The entries are
also of the order $\log m$. While this gives a
very significant improvement over the signature vector from the proof of
Theorem~\ref{maxtrans}, it still does not quite give a polynomial
number of signature vectors. However, as in~\cite{Annamalai2016,
  Annamalai2017}  it can be shown with a suitable alteration of the
above definition of signature vector (Definition~\ref{sig}), each
signature vector can be associated with a subset of integers from 1 to
$x$, where $x$ is of order $\log m$. It then follows that the
number of signature vectors, and hence the number of iterations of the
algorithm, is polynomial in $m$ (Lemma~\ref{fewsigs}). 

\section{Preliminaries}\label{prelims}
In this section we formalise the main notions we will need. Much of
the terminology in this section follows that
of~\cite{Annamalai2016, Annamalai2017}. Let $G$ and $(V_1,\dots,V_m)$
be as in  Section~\ref{over}. For the
definitions that follow, consider a PIT $M$ of $G$. 

\begin{definition}\label{blocks}
A vertex $u$ {\it blocks} a vertex $v$ if $u\in M$ and $uv\in E(G)$.
\end{definition}

\begin{definition}\label{immaddable}
A vertex $v$ is {\it immediately addable} with respect to $M$ if $v\notin M$ and it has no vertices in $V(G)$ blocking it. For $W\subseteq V(G)$, $I_M(W)$ denotes the set of vertices in $W$ that are immediately addable with respect to $M$.
\end{definition}

\begin{definition}\label{layer}
A {\it layer} $L$ of $G$ with respect to a PIT $M$ is a pair $(X,Y)$ where:
\begin{enumerate}
\item $X\subseteq V(G)\setminus M$,
\item $X$ is an independent set,
\item $Y\subseteq M$ is the set of blocking vertices of $X$, and
\item every $u\in Y$ is adjacent to exactly one vertex from $X$.
\end{enumerate}
\end{definition} For an example of a layer, refer to Figure~\ref{treealg}. Note that $Y$ is also an independent set since $M$ is an independent set.

\begin{definition}\label{tree}
Let $M$ be a PIT in $G$ and let $A$ be a vertex class in the vertex partition of $G$ that does not contain a vertex in $M$. An {\it alternating tree} $T$ with respect to $M$ and $A$ is a tuple $(L_0,\dots,L_\ell)$ where $\ell\ge0$ such that:
\begin{enumerate}
\item $L_0=(X_0,Y_0)=(\emptyset,\emptyset)$ and $A(Y_0):=A$,
\item $L_i=(X_i,Y_i)$ is a layer for each $1\le i\le\ell$,
\item $X_1\subseteq A$ and $X_i\subseteq \bigcup\limits_{v\in Y_{i-1}} A(v)$ for all $i=2,\dots,\ell$, and
\item $(X_i\cup Y_i)\cap (X_{i'}\cup Y_{i'})=\emptyset$ for all $i,i'\in\{0,\dots,\ell\}$, $|i-i'|>0$.
\end{enumerate}
We call $A$ the {\it root} of $T$. 
\end{definition}
Figure~\ref{treealg} provides an example of an alternating tree.

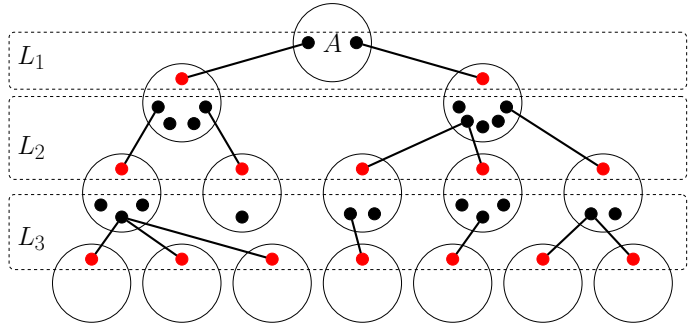
\begin{figure}[!ht]
\centering
\scalebox{.4}{
\begin{tikzpicture}
\node[font=\fontsize{30}{0}] (A) at (1,0) {$A$};  
\drawNode{1}{0}{-180,0}{}{a}
\drawNode{-4}{-2}{-10,-170,-120,-60}{90}{b}
\drawNode{6}{-2}{-170,-130,-90,-50,-10}{90}{c}
\drawNode{-6}{-5}{-150,-90,-30}{90}{d}
\drawNode{-2}{-5}{-90}{90}{e}
\drawNode{2}{-5}{-120,-60}{90}{f}
\drawNode{6}{-5}{-150,-90,-30}{90}{g}
\drawNode{10}{-5}{-120,-60}{90}{h}
\drawNode{-7}{-8}{}{90}{i}
\drawNode{-4}{-8}{}{90}{j}
\drawNode{-1}{-8}{}{90}{k}
\drawNode{2}{-8}{}{90}{l}
\drawNode{5}{-8}{}{90}{m}
\drawNode{8}{-8}{}{90}{n}
\drawNode{11}{-8}{}{90}{o}

\draw[line width=2pt] (Ba--180)--(Rb-90);
\draw[line width=2pt] (Ba-0)--(Rc-90);
\draw[line width=2pt] (Bb--170)--(Rd-90);
\draw[line width=2pt] (Bb--10)--(Re-90);
\draw[line width=2pt] (Bc--130)--(Rf-90);
\draw[line width=2pt] (Bc--130)--(Rg-90);
\draw[line width=2pt] (Bc--10)--(Rh-90);
\draw[line width=2pt] (Bd--90)--(Ri-90);
\draw[line width=2pt] (Bd--90)--(Rj-90);
\draw[line width=2pt] (Bd--90)--(Rk-90);
\draw[line width=2pt] (Bf--120)--(Rl-90);
\draw[line width=2pt] (Bg--90)--(Rm-90);
\draw[line width=2pt] (Bh--120)--(Rn-90);
\draw[line width=2pt] (Bh--120)--(Ro-90);

\node[font=\fontsize{30}{0}] (L1) at (-9,-.5) {$L_1$};   
\node[font=\fontsize{30}{0}] (L2) at (-9,-3.5) {$L_2$};
\node[font=\fontsize{30}{0}] (L3) at (-9,-6.5) {$L_3$};     
\node (K1) at (12.5,-.5) {};   
\node (K2) at (12.5,-3.5) {};
\node (K3) at (12.5,-7) {};
\node[draw, dashed, rounded corners=5pt, inner sep=4pt, fit= (L1) (Ba--180) (Rb-90) (K1)] {};
\node[draw, dashed, rounded corners=5pt, inner sep=4pt, fit=(L2) (Bb--170) (Rd-90) (K2)] {};
\node[draw, dashed, rounded corners=5pt, inner sep=4pt, fit=(L3) (Bd--150) (Ri-90) (K3)] {};
\end{tikzpicture}}
\caption{An alternating tree $T$ arising from a graph $G$ and PIT $M$. The circles are vertex classes, and vertices in the same layer of $T$ are enclosed with a dotted border. Within each layer $L_i$, the vertices in $X_i$ are shown in black and the vertices in $Y_i$ are shown in red.}
\label{treealg}
\end{figure}

Let $T=(L_0,\dots,L_\ell)$ be an alternating tree of $G$ with respect
to a PIT $M$ and root $A$. For each $0\leq j\leq\ell$ we
let $X_{\le j}=\bigcup\limits_{i=0}^{j}X_i$. Similarly, we
define $Y_{\le j}=\bigcup\limits_{i=0}^{j}Y_i$. Note that
$A(Y_{\leq\ell})$ is the set of vertex classes intersecting $T$. It
follows from Definitions~\ref{layer} and~\ref{tree} that the
    subgraph of $G$ induced by $X_{\le\ell}\cup Y_{\le\ell}\setminus
    I_M(X_{\le\ell})$ is a constellation for $A(Y_{\leq\ell})$ (see
    Figures~\ref{const} and~\ref{treealg}).

Our algorithm will make use of fixed constants $\mu$, $U$, and $\rho$
which will be chosen in advance and depend only on the input constants
$r$ and $\epsilon$. The following notion formalises a suitable choice.

\begin{definition}\label{feasible}
Let $r\ge2$ and $\epsilon>0$ be given. We say a tuple $(\mu, U, \rho)$ of positive real numbers is {\it feasible} for $(r,\epsilon)$ if the following hold:
\begin{enumerate}
\item $(2+\epsilon)\left[1-\frac{1}{U}\left(\frac{1+\mu U}{1-\mu}+\rho\right)\right]>\left(\frac{2+\mu(r+2)+\rho(r+1)}{1-\mu}\right)$,
\item $\epsilon\left[1-\frac{1}{U}\left(\frac{1+\mu
    U}{1-\mu}+\rho\right)\right]>\frac{\mu(r+4)+\rho(r+2)}{1-\mu}$, and 
\item $U-\mu\rho>\rho$.
\end{enumerate}
\end{definition}

For example, $(\mu,
U,\rho)=\left(\frac{\epsilon}{10r},\frac{10r}{\epsilon},\frac{\epsilon}{10r}\right)$
is feasible for $(r,\epsilon)$ when $r\ge2$ and ${0<\epsilon<1}$. Thus when $r$ and $\epsilon$ are fixed, the parameters $\mu$, $U$ and $\rho$ may also be taken to be fixed constants. As
mentioned in the introduction, we
make no attempt here to choose the constants to optimize the running
time (see Lemma~\ref{fewsigs}).

The following two definitions will depend on the fixed constants $\mu$ and $U$. Definition~\ref{addabledef} will apply when $X,Y$ forms a partially built layer $L_{\ell+1}$. 

\begin{definition}\label{addabledef}
Let $T=(L_0,\dots,L_\ell)$ be an alternating tree of $G$ with respect to a PIT $M$ and root $A$ and let $X,Y\subseteq V(G)$. A vertex $v\in A(Y_\ell)$ is an {\it addable} vertex for $X$, $Y$, and $T$ if $v\notin Y_{\ell}\cup X\cup Y$, $|A(v)\cap X|<U$, and there does not exist a vertex $u\in X_{\le\ell}\cup Y_{\le\ell}\cup X\cup Y$ such that $uv\in E(G)$.
\end{definition}

\begin{definition}\label{collapsible}
A layer $L_i=(X_i,Y_i)$ is {\it collapsible} if $I_M(X_i)>\mu|X_i|$.
\end{definition}

\section{Algorithms}\label{algs}
Recall that $G$ is an $r$-claw-free graph with respect to vertex
partition $(V_1,\dots,V_m)$. Let $M$ be a PIT in $G$ and let $A$ be a
vertex class in the vertex partition of $G$ that does not contain a
vertex in $M$. The main idea of the algorithm FindITorBD in
Theorem~\ref{main} is to perform a series of modifications to $M$ that
will allow us to augment it with a vertex in $A$. If we are not
successful then we will find a subset of classes (based on an
alternating tree of $G$) that has a small dominating set.

The algorithm FindITorBD is given in Section~\ref{sigs}. In the following subsections, we describe three algorithms that are used by FindITorBD. The first two algorithms, called BuildLayer and SuperposedBuild, are used as subroutines in the third algorithm, called GrowTransversal. GrowTransversal appears as the main subroutine of FindITorBD.

Let
$T=(L_0,\dots,L_\ell)$ be an alternating tree of $G$ with respect to
$M$ and $A$.
Let $\mu$, $U$, and $\rho$ be fixed constants chosen in advance so that $(\mu, U,\rho)$ is feasible for $(r,\epsilon)$ (see Definition~\ref{feasible}).

\subsection{BuildLayer}

BuildLayer is a subroutine in the main algorithm for augmenting $M$ that helps construct new layers for an alternating tree $T=(L_0,\ldots,L_{\ell})$. The function takes as inputs $T$ and some $X,Y\subseteq V(G)$ ($(X,Y)$ can be thought of as a \lq\lq partially built\rq\rq\hspace{1pt} layer). It then creates a new layer $L_{\ell+1}=(X_{\ell+1},Y_{\ell+1})$ by augmenting $X$ and $Y$ and returning the resulting pair $(X,Y)$.

\begin{algorithmic}[1]
\Function{BuildLayer}{$T,X,Y$}
\While{there is a vertex $v\in A(Y_\ell)$ that is addable for $X$, $Y$, and $T$}
	\State $X:=X\cup \{v\}$
	\State $Y:=Y\cup\{u\in M:uv\in E(G)\}$
\EndWhile
\State \Return $(X,Y)$
\EndFunction
\end{algorithmic}

\subsection{SuperposedBuild}

SuperposedBuild is a subroutine in the main algorithm for augmenting $M$ that, after a modification of $M$ occurs in the algorithm, modifies $T$ so that it remains an alternating tree with respect to the new PIT $M$. SuperposedBuild possibly augments $T$ by adding some vertices that are no longer blocked due to the modification of $M$. The function takes as inputs the current $T$ and its number of layers $\ell$. It then performs some tests on the layers of $T$, to see if any $X_i$ could be substantially enlarged, and returns a possibly modified alternating tree to replace $T$ for the next iteration of the main algorithm as well as the number of layers in the returned alternating tree.

\begin{algorithmic}[1]
\Function{SuperposedBuild}{$T,\ell$}
\State $i:=1$
\While{$i\le\ell$}
	\State $(X'_{i},Y'_{i}):=\text{BuildLayer}((L_0,\dots,L_{i-1}),X_{i},Y_{i})$
	\If{$|X'_{i}|\ge(1+\mu)|X_{i}|$}
		\State $(X_{i},Y_{i}):=(X'_{i},Y'_{i})$
		\State $L_i:=(X_i,Y_i)$
		\State $T:=(L_0,\dots,L_i)$
		\State $\ell:=i$
	\EndIf
	\State $i:=i+1$
\EndWhile
\State \Return $(T,\ell)$
\EndFunction
\end{algorithmic}

\subsection{GrowTransversal}
GrowTransversal is the main algorithm for augmenting $M$. It takes as
inputs $M$ and $A$ and performs a series of modifications to $M$ until
either a vertex in $A$ is added to $M$ or an iteration constructs a
layer $L_i$ with too small an $X_i$ relative to the size of $T$. When
GrowTransversal terminates, it returns $M$, $T$, and a flag variable
$x$ as $(M,T,x)$. The variable $x$ is set to $1$ if GrowTransversal
terminates due to an iteration constructing a layer $L_i$ with at most
$\rho|Y_{\le i-1}|$ vertices in $X_i$ and is set to $0$ if
GrowTransversal successfully augments $M$ with a vertex in $A$. If
GrowTransversal returns $(M,T,1)$, we will show in the next section
that $T$ contains a subset $\mathcal B$ of vertex classes whose
vertices are dominated by a set of fewer than ${(2+\epsilon)|{\mathcal
    B}|}$ vertices with the properties stated in Theorem~\ref{main}. 

\begin{algorithmic}[1]
\Function{GrowTransversal}{$M,A$}
\State $L_0:=(X_0,Y_0):=(\emptyset, \emptyset)$
\State $A(Y_0):=A$
\State $L_0:=(X_0,Y_0)$
\State $T:=(L_0)$
\State $\ell:=0$
\While{$A\cap M=\emptyset$}
	\State $(X_{\ell+1},Y_{\ell+1}):=\text{BuildLayer}(T, \emptyset, \emptyset)$
	\State $L_{\ell+1}:=(X_{\ell+1},Y_{\ell+1})$
	\State $T:=(L_0,\dots,L_\ell,L_{\ell+1})$
	
	\If{$|X_{\ell+1}|\le\rho|Y_{\le\ell}|$}
		\State \Return $(M,T,1)$ and terminate
	\Else
		\State $\ell:=\ell+1$

		\While{$|I_M(X_\ell)|>\mu|X_\ell|$}
			\If{$\ell=1$}
				\State Augment $M$ with a vertex from $I_M(X_1)$ 
				\State \Return $(M,T,0)$ and terminate
			\Else
				\ForAll{$w\in Y_{\ell-1}$ such that $I_M(X_\ell)\cap A(w)\neq\emptyset$}
					\State $M:=(M\setminus\{w\})\cup\{u\}$ for some arbitrary ${u\in I_M(X_\ell)\cap A(w)}$
					\State $Y_{\ell-1}:=Y_{\ell-1}\setminus\{w\}$
				\EndFor
			\EndIf
			\State $T':=(L_0,\dots,L_{\ell-1})$
			\State $\ell':=\ell-1$
			\State $(T,\ell):=\text{SuperposedBuild}(T',\ell')$			
		\EndWhile
	\EndIf
\EndWhile
\EndFunction
\end{algorithmic}

The GrowTransversal algorithm begins by initializing the alternating tree $T$ with respect to $M$ and its number of layers $\ell$. While $A$ does not contain a vertex in the PIT $M$, the algorithm repeats a building layer operation (line 8) followed by a loop of {\it collapsing} operations (lines 15-28) that modify $M$ when enough immediately addable vertices with respect to $M$ are present in the newly constructed layer. Figure~\ref{collapsefig} shows an example of one collapse operation (lines 20-23).  

\begin{figure}[!ht]
\centering
\subfloat[]{\scalebox{.249}{
\begin{tikzpicture}
\drawNode{-4}{-2}{-170,-90,-10}{90}{b}
\drawNode{6}{-2}{-170,-120,-60,-10}{90}{c}
\drawNode{-6}{-5}{-170,-90,-10}{90}{d}
\drawNode{-2}{-5}{-170,-10}{90}{e}
\drawNode{2}{-5}{-170,-90}{90}{f}
\drawNode{6}{-5}{-90}{90}{g}
\drawNode{10}{-5}{-90}{90}{h}
\drawNode{-7}{-9}{}{90}{i}
\drawNode{-4}{-9}{}{90}{j}
\drawNode{-1}{-9}{}{90}{k}
\drawNode{2}{-9}{}{90}{l}
\drawNode{6}{-9}{}{90}{m}
\drawNode{9}{-9}{}{90}{n}
\drawNode{12}{-9}{}{90}{o}

\draw[line width=2pt] (-2,-.5)--(Rb-90);
\draw[line width=2pt] (4,-0.5)--(Rc-90);
\draw[line width=2pt] (Bb--170)--(Rd-90);
\draw[line width=2pt] (Bb--10)--(Re-90);
\draw[line width=2pt] (Bc--120)--(Rf-90);
\draw[line width=2pt] (Bc--130)--(Rg-90);
\draw[line width=2pt] (Bc--10)--(Rh-90);
\draw[line width=2pt] (Bd--90)--(Ri-90);
\draw[line width=2pt] (Bd--90)--(Rj-90);
\draw[line width=2pt] (Bd--90)--(Rk-90);
\draw[line width=2pt] (Bf--90)--(Rl-90);
\draw[line width=2pt] (Bg--90)--(Rm-90);
\draw[line width=2pt] (Bh--90)--(Rn-90);
\draw[line width=2pt] (Bh--90)--(Ro-90);

\node[draw, dashed, rounded corners=5pt, inner sep=4pt, fit= (Bd--170) (Bf--170)] {};
\end{tikzpicture}}}\hspace{1mm}
\subfloat[]{\scalebox{.249}{
\begin{tikzpicture}
\drawNode{-4}{-2}{-170,-90,-10}{90}{b}
\drawNode{6}{-2}{-170,-120,-60,-10}{90}{c}
\drawNode{-6}{-5}{-170,-90}{-10}{d}
\drawNode{-2}{-5}{-170}{-10}{e}
\drawNode{2}{-5}{-90}{-170}{f}
\drawNode{6}{-5}{-90}{90}{g}
\drawNode{10}{-5}{-90}{90}{h}
\drawNode{-7}{-9}{}{90}{i}
\drawNode{-4}{-9}{}{90}{j}
\drawNode{-1}{-9}{}{90}{k}
\drawNode{2}{-9}{}{90}{l}
\drawNode{6}{-9}{}{90}{m}
\drawNode{9}{-9}{}{90}{n}
\drawNode{12}{-9}{}{90}{o}

\draw[line width=2pt] (-2,-.5)--(Rb-90);
\draw[line width=2pt] (4,-0.5)--(Rc-90);
\draw[line width=2pt] (Bc--120)--(Rg-90);
\draw[line width=2pt] (Bc--10)--(Rh-90);
\draw[line width=2pt] (Bd--90)--(Ri-90);
\draw[line width=2pt] (Bd--90)--(Rj-90);
\draw[line width=2pt] (Bd--90)--(Rk-90);
\draw[line width=2pt] (Bf--90)--(Rl-90);
\draw[line width=2pt] (Bg--90)--(Rm-90);
\draw[line width=2pt] (Bh--90)--(Rn-90);
\draw[line width=2pt] (Bh--90)--(Ro-90);
\end{tikzpicture}}}
\subfloat[]{\scalebox{.249}{
\begin{tikzpicture}
\drawNode{-4}{-2}{-170,-90,-10}{90}{b}
\drawNode{6}{-2}{-170,-120,-60,-10}{90}{c}
\hideNode{-6}{-5}{-170,-90}{-10}{d}
\hideNode{-2}{-5}{-170}{-10}{e}
\hideNode{2}{-5}{-90}{-170}{f}
\drawNode{6}{-5}{}{90}{g}
\drawNode{10}{-5}{}{90}{h}
\hideNode{-7}{-9}{}{90}{i}
\hideNode{-4}{-9}{}{90}{j}
\hideNode{-1}{-9}{}{90}{k}
\hideNode{2}{-9}{}{90}{l}
\hideNode{6}{-9}{}{90}{m}
\hideNode{9}{-9}{}{90}{n}
\hideNode{12}{-9}{}{90}{o}

\draw[line width=2pt] (-2,-.5)--(Rb-90);
\draw[line width=2pt] (4,-0.5)--(Rc-90);
\draw[line width=2pt] (Bc--120)--(Rg-90);
\draw[line width=2pt] (Bc--10)--(Rh-90);
\end{tikzpicture}}}
\caption{An example of collapsing layer $L_\ell$. In (a), a substantial set $I_M(X_{\ell})$ of immediately addable vertices is found in $X_{\ell}$ (line 15). In (b), the vertex $w$ of $M$ in each class that intersects $I_M(X_{\ell})$ is replaced by some $u\in I_M(X_{\ell})$ in that same class (line 21). In (c), each $w\in M$ that was replaced in line 21 disappears from $Y_{\ell-1}$ (line 22) and the entire last layer $(X_{\ell},Y_{\ell})$ is removed (line 25).}
\label{collapsefig}
\end{figure}
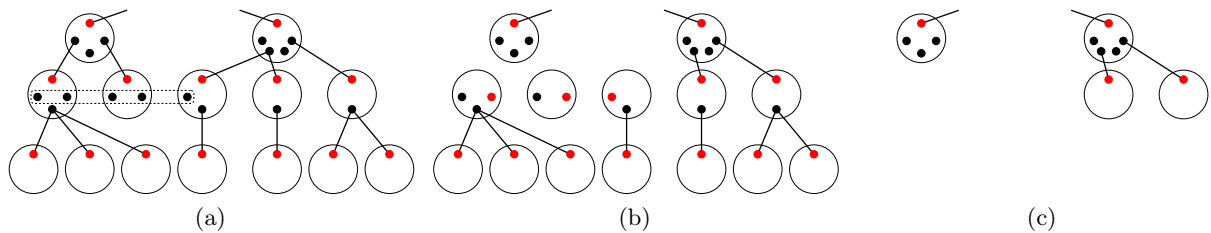

Performing one collapse operation can result in more vertices becoming immediately addable with respect to the new $M$ in earlier layers. Hence one collapse operation can lead to a cascade of collapse operations being performed on layers of $T$. Also, any collapse operation can modify $M$ in such a way that SuperposedBuild can modify a layer $(X_i,Y_i)$ of $T$ to contain a significantly larger number of vertices ($|X'_i|\ge(1+\mu)|X_i|$ and so replacing $(X_i,Y_i)$ by $(X'_i,Y'_i)$ adds at least $\mu|X_i|$ vertices to the layer). These modifications to a layer (which is the final layer of the tree returned by SuperposedBuild) may still leave it collapsible, so more collapse operations may be possible if SuperposedBuild modifies a layer of $T$. Thus, for the $T$ at the conclusion of line 28 to remain an alternating tree with respect to $M$ after all these modifications are made to $M$, GrowTransversal removes all layers of $T$ constructed after the earliest layer that contains a vertex class where $M$ is modified (lines 25-27). This leaves the resulting $T$ as an alternating tree with respect to the modified $M$. 

Due to these operations, at the beginning and end of each iteration of the main loop of GrowTransversal (which starts in line 7), $M$ remains a PIT and $T$ remains an alternating tree with respect to $M$. Also, the modifications to $M$ do not change the set of vertex classes containing vertices in $M$ (lines 20-23) unless $L_1$ is collapsible. In this case, a vertex in $I_M(X_1)$, which is therefore in $A$, is added to $M$ (lines 15-18) and so $A$ is added to the set of vertex classes covered by the PIT $M$. As this is the goal of GrowTransversal, the algorithm returns $(M,T,0)$ and terminates (see line 18).

The algorithm also terminates if, at the start of some iteration,
BuildLayer produces a layer $L_{\ell+1}$ whose $X_{\ell+1}$ is
not sufficiently large with respect to the number of vertex classes
already in $T$, i.e. $|X_{\ell+1}|\le\rho|Y_{\le\ell}|$ (see lines
11-12). If GrowTransversal terminates because of this, GrowTransversal
returns $(M,T,1)$, which distinguishes this case from when the
algorithm terminates because it successfully augments $M$ to include a
vertex of $A$. We will show in the next section that given an
alternating tree $T=(L_0,\dots, L_{\ell+1})$ with respect to $M$ and
$A$ such that $|X_{\ell+1}|\le\rho|Y_{\le\ell}|$, there exists some
set $\cB$ of the vertex classes in $T$ such that $G_\cB$ is
dominated by a set $D$ of vertices with the properties stated in
outcome (2) of Theorem~\ref{main}. Our analysis will provide a specific $\cB$
and its corresponding $D$ given that GrowTransversal returns
$(M,T,1)$.

\section{Analysis}\label{analysis}
The main result of this section (Lemma~\ref{allbigx}) shows that if
GrowTransversal terminates because BuildLayer constructs a layer
$L_{\ell+1}$ with $|X_{\ell+1}|\leq\rho|Y_{\leq\ell}|$ (line 11), then
we can find a set $\cB$ of vertex classes and a set $D$ of vertices
that satisfy the conditions of outcome (2) of
Theorem~\ref{main}. Otherwise GrowTransversal will succeed in
augmenting the current PIT $M$ with a vertex of $A$.

As discussed in the Overview (Section~\ref{over}), we will bound the
total number of steps taken by 
GrowTransversal using a {\it signature vector} (see
Definition~\ref{sig}) that is defined in
terms of logarithmic functions of the sizes of the layers of $T$. The last
lemma of this section (Lemma~\ref{fewlayers}) establishes that the
number of layers in $T$ (and hence length of any 
signature vector) is bounded by a logarithmic function of the number
$m$ of vertex classes. This fact will be key to the proof given in
Section~\ref{sigs} that the total number of signature vectors (and hence the
number of steps taken by GrowTransversal) is polynomial in $m$. 

Recall from Section~\ref{prelims} that we can assume $G$ is an
$r$-claw-free graph with respect to vertex partition
$(V_1,\dots,V_m)$, each vertex class $V_i$ is an independent set of
vertices, and $r\ge2$ and $\epsilon>0$ are fixed. Also, recall from
Section~\ref{algs} that $\mu$, $U$, and $\rho$ are fixed constants
such that $(\mu,U,\rho)$ is feasible for $(r,\epsilon)$. 

For concreteness, we may also assume that the vertices of the graph
$G$ have been assigned some arbitrary but fixed ordering, and that
vertices are processed by our algorithms subject to this ordering. For
example, we assume that
$\text{BuildLayer}(T,X,Y)$ adds vertices to $X$ and $Y$ in order,
i.e. during an iteration of the while loop in 
BuildLayer, the addable vertex with the lowest index in the ordering
is the vertex chosen to be $v$ in line 3. Similarly we may assume that the
vertex $u$ in line 21 of GrowTransversal is chosen to be the vertex in
$I_M(X_\ell)\cap A(w)$ with the lowest index in the ordering. The
proof of Lemma~\ref{allbigx} will also use this convention.

We begin by establishing two preliminary results on basic properties of
the alternating tree constructed in GrowTransversal.

\begin{lemma}\label{nocollapse}
Let $T=(L_0,\dots, L_\ell)$ and $M$ be the alternating tree and partial independent transversal at the beginning of some iteration of the while loop in line $7$ of {\rm GrowTransversal}. Then none of the layers $L_0,\dots, L_\ell$ are collapsible. Hence $|Y_i|\ge(1-\mu)|X_i|$ for each $i\in\{1,\dots,\ell\}$.
\end{lemma}
\begin{proof}[Proof:]
Suppose the statement holds at the beginning of the current iteration. In lines $8$ and $9$ of GrowTransversal, $L_{\ell+1}$ is constructed. If GrowTransversal does not terminate and $L_{\ell+1}$ is not collapsible, then the claim follows for the beginning of the next iteration since none of the earlier layers were modified in the current iteration. 

If GrowTransversal does not terminate and $L_{\ell+1}$ is collapsible, then let ${T:=(L_0,\dots,L_k)}$ be the alternating tree with respect to $M$ that results from the loop of collapsing operations where $k\le\ell$ (i.e. $T$ is the alternating tree after lines 15-28 are completed). Unless the algorithm terminates, $k\ge1$ (lines 15-18). 

Each time a collapse operation is performed (lines 20-23), the number of layers in $T$ is reduced (line 25). Also, if SuperposedBuild modifies $T$ in line 27, the number of layers of $T$ is reduced or stays the same. Thus, layers $L_0,\dots,L_{k-1}$ are not changed by the loop of collapsing operations and so remain unchanged throughout the current iteration. Layer $L_k$ may be modified in lines 21, 22, and 27. However, $L_k$ cannot be collapsible since it is the final layer in $T$ after the loop of collapsing operations terminates. Hence none of the layers in an alternating tree are collapsible at the end of an iteration of the while loop in line 7 (unless the algorithm terminates during the iteration).

Since the claim holds for the first iteration of the while loop in line 7, the statement follows by induction on the number of iterations of this loop in GrowTransversal. As $L_i$ is a layer for each $i\in\{1,\dots,\ell\}$, $Y_i$ contains all blocking vertices of all vertices in $X_i$ and every vertex in $Y_i$ is adjacent to exactly one vertex in $X_i$ by construction. Thus, there are at most $|Y_i|$ vertices in $X_i\setminus I_M(X_i)$ and at most $\mu|X_i|$ vertices in $I_M(X_i)$ for each layer $i=1,\dots,\ell$. Hence $|X_i|\le |Y_i|+\mu|X_i|$ for each $i\in\{1,\dots,\ell\}$.\end{proof}

\begin{lemma}\label{nosuper}
Let $T=(L_0,\dots, L_\ell)$ and $M$ be the alternating tree and partial independent transversal at the beginning of some iteration of the while loop in line $7$ of {\rm GrowTransversal}. Then for each $i\in\{1,\dots,\ell\}$, $$ (X'_i,Y'_i):=\text{BuildLayer}((L_0,\dots,L_{i-1}),X_i,Y_i)$$ satisfies $|X'_i|<(1+\mu)|X_i|$.
\end{lemma}
\begin{proof}[Proof:]
Consider layer $L_i$ at the beginning of the current iteration for
some $1\leq i\leq\ell$. During the iteration that $L_i$ was
constructed by either BuildLayer in line 8 or SuperposedBuild in line
27, an (additional) application of SuperposedBuild could not increase
the size of $X_i$ (both functions created $X_i$ to be as large as
possible with respect to $M$ in that iteration).  

Suppose no layer built during the iterations between when $L_i$ was constructed and the current iteration is collapsible, i.e. the condition of line 15 is not met between the iteration $L_i$ was constructed and the current iteration of GrowTransversal. Then, since $M$ is not changed during the intervening iterations, SuperposedBuild does not increase the size of $X_i$. 

Now suppose some layer built during the intervening iterations was collapsible. Note that the index of the collapsible layer must be greater than $i$ as otherwise $L_i$ would be discarded (lines 25-27). Thus, for each $j\ge i+1$ such that $L_j$ is collapsible, SuperposedBuild tries to augment $X_i$ by at least a $\mu$ proportion of its size. However, since $L_i$ is a layer in $T$ at the start of the current iteration, SuperposedBuild does not succeed in changing $X_i$. Hence BuildLayer$((L_0,\dots,L_{i-1}),X_i,Y_i)$ does not increase the number of vertices in $X_i$ by $\mu|X_i|$.
\end{proof}

We are now ready to prove the main result of this section.

\begin{lemma}\label{allbigx}
Assume $(\mu,U,\rho)$ is feasible for $(r,\epsilon)$. Let
$T=(L_0,\dots, L_\ell)$ and $M$ be the alternating tree with root
  vertex $A$ and partial independent transversal at the beginning of
some iteration of the while loop in line $7$ of {\rm
  GrowTransversal}. Then either 
\begin{enumerate}
\item when $L_{\ell+1}$ is constructed in line 8 of
{\rm GrowTransversal} we have $|X_{\ell+1}|>\rho|Y_{\le\ell}|$, or
\item {\rm GrowTransversal} terminates in line 12, and there exists a subset $\mathcal B$ of the set $\cB_0=A(Y_{\le\ell})$ such 
that $G_{\cB}$ is dominated by a set $D$ of vertices in $G$ of size less than
$(2+\epsilon)(|{\mathcal B}|-1)$. Moreover $K=G[X_{\leq\ell}\cup
  Y_{\leq\ell}\setminus I_M(X_{\leq\ell})]$ is a constellation for
$\cB_0$ and $D$ contains $V(K)$, where $|D\setminus V(K)|<\epsilon(|\cB|-1)$. 
\end{enumerate}
\end{lemma}
\begin{proof}[Proof:]
Suppose that after $L_{\ell+1}$ is constructed in line 8,
$|X_{\ell+1}|\le\rho|Y_{\le\ell}|$. For each $1\le i\le\ell$, let
${(X'_i,Y'_i)=\text{BuildLayer}((L_0,\dots,L_{i-1}),X_i,Y_i)}$. 
Define $\mathcal B$ algorithmically by performing the following steps in order: 
\begin{enumerate}
\item[(i)] ${\mathcal B}:=\cB_0=A(Y_{\le\ell})$.
\item[(ii)] Remove all vertex classes of ${\mathcal B}$ that contain $U$ vertices in $X_{\le\ell+1}$.
\item[(iii)] Remove all vertex classes in $\bigcup\limits_{i=1}^\ell A(X'_i\setminus X_i)$.
\end{enumerate}
Clearly $\cB\subseteq\cB_0$, and
$K$ is a constellation for $\cB_0$ (as noted after
Definition~\ref{tree}). 
\setcounter{claim}{0}
\begin{claim}\label{Bsize} We have
$$|\cB|-1\geq\left[1-\frac{1}{U}\left(\frac{1+\mu
      U}{1-\mu}+\rho\right)\right]|Y_{\le\ell}|.$$
\end{claim}
\begin{proof}[Proof:]
The vertex classes in $\mathcal B$ 
include the vertex classes of $T$ that do not contain $U$ vertices in
$X_{\le\ell+1}$ and do not contain any addable vertices for $X_i$,
$Y_i$, and $(L_0,\dots,L_{i-1})$ for all $1\le i\le\ell$. We use these
facts to bound $|\mathcal B|$ from below as follows.

Recall that $|A(Y_{\le\ell})|=|Y_{\le\ell}|+1$. The set $A^U$ of vertex
classes in $A(Y_{\le\ell})$ that contain $U$ vertices in
$X_{\le\ell+1}$ has size at most $U^{-1}|X_{\le\ell+1}|$. By
Lemma~\ref{nosuper}, $|X'_i|<(1+\mu)|X_i|$ for each
$i\in\{1,\dots,\ell\}$. As $X_i\subseteq X'_i$, this implies that
there are at most $\mu|X_i|$ vertices in $X'_i\setminus X_i$. Thus
$|A(X'_i\setminus X_i)|\le\mu|X_i|$ for all $1\le i\le\ell$ and so
$\sum\limits_{i=1}^{\ell} |A(X'_i\setminus
X_i)|\le\mu|X_{\le\ell}|$. Also by Lemma~\ref{nocollapse},
$|X_{\le\ell}|\le\frac{1}{1-\mu}|Y_{\le\ell}|$ and by the assumption,
$|X_{\ell+1}|\le\rho|Y_{\le\ell}|$. Therefore,  
\begin{align*}
|{\mathcal B}|&\ge|A(Y_{\le\ell})|-|A^U|-\left|\bigcup\limits_{i=1}^\ell A(X'_i\setminus X_i)\right|\\
&\ge|A(Y_{\le\ell})|-\frac{1}{U}|X_{\le\ell+1}|-\mu|X_{\le\ell}|\\
&=(|Y_{\le\ell}|+1)-\left(\frac{1}{U}|X_{\le\ell}|+\frac{1}{U}|X_{\ell+1}|+\mu|X_{\le\ell}|\right)\\
&\ge |Y_{\le\ell}|+1-\left[\left(\frac{1}{U}+\mu\right)|X_{\le\ell}|+\frac{\rho}{U}|Y_{\le\ell}|\right]\\
&\ge |Y_{\le\ell}|+1-\left[\left(\frac{1}{U}+\mu\right)\left(\frac{1}{1-\mu}\right)|Y_{\le\ell}|+\frac{\rho}{U}|Y_{\le\ell}|\right]\\
&= 1+\left[1-\frac{1}{U}\left(\frac{1+\mu U}{1-\mu}+\rho\right)\right]|Y_{\le\ell}|.
\end{align*}
\end{proof}
Let $B$ denote the set of vertices in the vertex classes of $\mathcal B$ and let $$W=X'_{\le\ell}\cup Y'_{\le\ell}\cup X_{\ell+1}\cup Y_{\ell+1}.$$
\begin{claim}\label{Wdom}
The set $W$ dominates $G[B\setminus I_M(W)]$.
\end{claim}
\begin{proof}[Proof:]
Let $u\in B\setminus I_M(W)$. By (iii) we have that $u\notin
X'_{\le\ell}\setminus X_{\le\ell}$.  

Suppose $u\in W\setminus I_M(W)$. Then $u\in (X_{\le\ell+1}\cup Y'_{\le\ell}\cup Y_{\ell+1})\setminus I_M(W)$. If $u\in Y'_i$ for some $i\in\{1,\dots,\ell\}$, then the construction of $(X'_i,Y'_i)$ and $(X_i,Y_i)$ by BuildLayer implies that $u$ has a neighbour $v$ in $X'_i$ (lines 3-4 of BuildLayer). Hence $v\in W$. Similarly, if $u\in Y_{\ell+1}$, then $u$ has a neighbour $v$ in $X_{\ell+1}$ (lines 3-4 of BuildLayer). If $u\in X_i$ for some $i\in\{1,\dots,\ell+1\}$, then since $u\notin I_M(W)$, $u$ has a neighbour $v$ that blocks $u$. By the construction of $(X_i,Y_i)$ (lines 2-4 of BuildLayer), $v\in Y_i$ and so $v\in W$. Therefore, every $u\in W\setminus I_M(W)$ has a neighbour in $W$. Thus we may assume $u\in B\setminus (W\cup I_M(W))$. 

Note that each vertex class in $\mathcal B$ has at most one vertex in $M$ and that these vertices are in $Y'_{\le\ell}$, hence $u\notin M$. Since $A(u)\in {\mathcal B}$, let $L_i$ be the layer of ${T'=(L_0,\dots,L_{\ell+1})}$ such that $A(u)\cap Y_{i-1}\neq\emptyset$ for some $1\le i\le \ell+1$.

Suppose $i<\ell+1$ and $u$ has no neighbours in $X'_{\le i}\cup Y'_{\le i}$. By (ii), $A(u)$ contains fewer than $U$ vertices in $X_i$. Also, by (iii), $A(u)$ contains no vertices in $X'_i\setminus X_i$. Thus by Definition~\ref{addabledef}, $u$ is an addable vertex for $X_i$, $Y_i$, and $(L_0,\dots,L_{i-1})$. Hence $\text{BuildLayer}((L_0,\dots,L_{i-1}),X_i,Y_i)$ would not stop until either $u$ is added to $X'_i$ or $u$ has a neighbour in $X'_i\cup Y'_i$. As $u\in B\setminus (W\cup I_M(W))$ and $X'_i\subseteq W$, we know that $u\notin X'_i$. Therefore $u$ has a neighbour in $X'_i\cup Y'_i$.

Now suppose $i=\ell+1$ and $u$ has no neighbours in $W$. Then again by
(ii) and (iii), $A(u)$ contains fewer than $U$ vertices in $X_i$ and
$A(u)$ contains no vertices in $X'_i\setminus X_i$. Thus by
Definition~\ref{addabledef}, $u$ is an addable vertex for
$X_{\ell+1}$, $Y_{\ell+1}$, and $(L_0,\dots,L_{\ell})$. Hence
$\text{BuildLayer}((L_0,\dots,L_{\ell}),\emptyset,\emptyset)$ would
not stop until either $u$ is added to $X_{\ell+1}$ or $u$ has a
neighbour in $X_{\ell+1}\cup Y_{\ell+1}$. As $u\in B\setminus (W\cup
I_M(W))$ and $X_{\ell+1}\subseteq W$, we know that $u\notin
X_{\ell+1}$. Therefore $u$ has a neighbour in $X_{\ell+1}\cup
Y_{\ell+1}$.

We therefore conclude that $W$ dominates  $G[B\setminus I_M(W)]$.
\end{proof}
Define $S$ to be the set of all $u\in V(G)$ for which there exists $v\in I_M(W)$ such that $u\in N(v)$ and $u$ is the neighbour of $v$ with the smallest index in
the ordering. 
\begin{claim}\label{Ssize}
The set $S$ dominates $I_M(W)$ and
$|S|\leq |I_M(W)|<\frac{2\mu+\rho}{1-\mu}|Y_{\le\ell}|$. 
\end{claim}
\begin{proof}[Proof:]
As each $v\in I_M(W)$ has at
least one neighbour in $G$, and the neighbour with the
smallest index in the ordering is in $S$, we have that $S$
dominates $I_M(W)$ and $|S|\le|I_M(W)|$. By Lemma~\ref{nosuper},
$|X'_{\le\ell}|<(1+\mu)|X_{\le\ell}|$ and so $|X'_{\le\ell}\setminus
X_{\le\ell}|<\mu|X_{\le\ell}|$. Also, by definition, $I_M(W)\subseteq
X'_{\le\ell}\cup X_{\ell+1}$ (see Definition~\ref{immaddable}). Thus  
\begin{align*}
|I_M(W)|&=|I_M(X'_{\le\ell}\cup X_{\ell+1})|\\
&=|I_M(X'_{\le\ell})|+|I_M(X_{\ell+1})|\\
&\le|I_M(X_{\le\ell})|+|I_M(X'_{\le\ell}\setminus X_{\le\ell})|+|X_{\ell+1}|\\
&<\mu|X_{\le\ell}|+\mu|X_{\le\ell}|+|X_{\ell+1}|\\
&=2\mu|X_{\le\ell}|+|X_{\ell+1}|.
\end{align*}
Recall that
${|X_{\le\ell}|\le\frac{1}{1-\mu}|Y_{\le\ell}|}$ by
Lemma~\ref{nocollapse}.
Since Conclusion (1) of Lemma~\ref{allbigx} does not hold we
know that $|X_{\ell+1}|\leq\rho|Y_{\le\ell}|$, and so we obtain
$$|I_M(W)|<\frac{2\mu}{1-\mu}|Y_{\le\ell}|+\rho|Y_{\le\ell}|,$$
from which the claim follows.
\end{proof}
Let $D=W\cup S.$ 
Then $D$ contains $V(K)$, and by Claim~\ref{Wdom} and the choice of $S$ we see
that $D$ dominates $G_{\cB}$. 

To help us estimate the size of $D$, we first establish the following.
\begin{claim}\label{Yprime} We have
$|Y'_{\le\ell}\setminus Y_{\le\ell}|<\mu(r-1)|X_{\le\ell}|$.
\end{claim}
\begin{proof}[Proof:]
Let $v\in Y'_{\le\ell}\setminus Y_{\le\ell}$. Then by
definition $v\in M$, and hence $v$ blocks its neighbours. If $uv\in E(G)$
for any $u\in X_{i}$ such that $1\le i\le\ell$,  $v$ blocks $u$ and so
would be included in $Y_i$ (line 4 of BuildLayer). This implies $v\in
Y_{\le\ell}$, which is a contradiction. Therefore $v$ is not adjacent
to any $u\in X_{\le\ell}$. However, $v\in Y'_i\setminus Y_i$ for some
$1\le i\le\ell$ and, by the construction of $(X'_i,Y'_i)$, $v$ is
adjacent to exactly one $u\in X'_i$ (lines 2-5). Thus $v$ has a
neighbour in $X'_i\setminus X_i$ and so $v$ has a neighbour in
$X'_{\le\ell}\setminus X_{\le\ell}$. 

Note that $Y'_i$ is a set of independent vertices in distinct vertex classes. As $G$ is $r$-claw-free, each vertex of $X'_i$ has at most $r-1$ independent neighbours in different vertex classes. Since $|X'_{\le\ell}\setminus X_{\le\ell}|<\mu|X_{\le\ell}|$, we have
$$|Y'_{\le\ell}\setminus Y_{\le\ell}|\le (r-1)|X'_{\le\ell}\setminus X_{\le\ell}|< \mu(r-1)|X_{\le\ell}|.$$
\end{proof}
It remains to bound $|D|$ and $|D\setminus V(K)|$. To do this, we note
that $D=X_{\le\ell}\cup Y_{\le\ell}\cup Q$ and $D\setminus
V(K)=Q\cup I_M(X_{\le\ell})$ where $$Q=(X'_{\le\ell}\setminus
X_{\le\ell})\cup        
(Y'_{\le\ell}\setminus Y_{\le\ell})\cup X_{\ell+1}\cup Y_{\ell+1}\cup
S.$$   
\begin{claim}\label{Qsize} We have
$$|Q|<\frac{\mu(r+2)+\rho(r+1)}{1-\mu}|Y_{\le\ell}|.$$
\end{claim}
\begin{proof}[Proof:]
Since $G$ is $r$-claw-free we know that $|Y_{\ell+1}|\le
  (r-1)|X_{\ell+1}|$. Since Conclusion
(1) of Lemma~\ref{allbigx} does not hold we know that
$|X_{\ell+1}|\le \rho|Y_{\le\ell}|$, and therefore
$|Y_{\ell+1}|\le\rho(r-1)|Y_{\le\ell}|$. 
We bound each of the remaining three summands below using (respectively)
Lemma~\ref{nosuper}, Claim~\ref{Yprime}, and Claim~\ref{Ssize}, to obtain 
\begin{align*}
|Q|&\leq|X'_{\le\ell}\setminus
X_{\le\ell}|+
|Y'_{\le\ell}\setminus Y_{\le\ell}|+ |X_{\ell+1}|+|Y_{\ell+1}|+
|S|\\
&\leq\mu|X_{\le\ell}|+\mu(r-1)|X_{\le\ell}|+\rho|Y_{\le\ell}|+\rho(r-1)|Y_{\le\ell}|+\frac{2\mu+\rho}{1-\mu}|Y_{\le\ell}|\\
&=\mu r|X_{\le\ell}|+\rho r|Y_{\le\ell}|+\frac{2\mu+\rho}{1-\mu}|Y_{\le\ell}|<\mu r|X_{\le\ell}|+\frac{2\mu+\rho(r+1)}{1-\mu}|Y_{\le\ell}|.
\end{align*}
Since
${|X_{\le\ell}|\le\frac{1}{1-\mu}|Y_{\le\ell}|}$ by
Lemma~\ref{nocollapse}, we conclude
$|Q|<\frac{\mu(r+2)+\rho(r+1)}{1-\mu}|Y_{\le\ell}|$ as required.
\end{proof}
Now Claim~\ref{Qsize} and Lemma~\ref{nocollapse} combine to give
\begin{align*}
|D|&=|X_{\le\ell}\cup Y_{\le\ell}\cup Q|\\
   &<\left(\frac1{1-\mu}+1+\frac{\mu(r+2)+\rho(r+1)}{1-\mu}\right)|Y_{\le\ell}|
<\left(\frac{2+\mu(r+2)+\rho(r+1)}{1-\mu}\right)|Y_{\le\ell}|.
\end{align*}

Claim~\ref{Bsize} and the fact that $(\mu,U,\rho)$ is feasible for
$(r,\epsilon)$ (Condition~(1) in Definition~\ref{feasible}) tell us
\begin{align*}
(2+\epsilon)(|{\mathcal
    B}|-1)&\ge(2+\epsilon)\left[1-\frac{1}{U}\left(
\frac{1+\mu U}{1-\mu}+\rho\right)\right]|Y_{\le\ell}|\\
&>\left(\frac{2+\mu(r+2)+\rho(r+1)}{1-\mu}\right)|Y_{\le\ell}|>|D|.
\end{align*}

To bound $|D\setminus V(K)|$ using Claim~\ref{Qsize}, we observe that
$I_M(X_{\le\ell})\subseteq I_M(W)$, so we may use 
Claim~\ref{Ssize} to conclude that 
\begin{align*}
|D\setminus V(K)|=|Q\cup I_M(X_{\le\ell})|&< \frac{\mu(r+2)+\rho(r+1)}{1-\mu}|Y_{\le\ell}|+\frac{2\mu+\rho}{1-\mu}|Y_{\le\ell}|\\
&=\frac{\mu(r+4)+\rho(r+2)}{1-\mu}|Y_{\le\ell}|.
\end{align*}

Using Claim~\ref{Bsize} again, and Condition~(2) of Definition~\ref{feasible} of
feasibility of $(\mu,U,\rho)$ we find
\begin{align*}
\epsilon(|{\mathcal
    B}|-1)&\ge\epsilon\left[1-\frac{1}{U}\left(\frac{1+\mu
      U}{1-\mu}+\rho\right)\right]|Y_{\le\ell}|\\
&>\frac{\mu(r+4)+\rho(r+2)}{1-\mu}|Y_{\le\ell}|
>|D\setminus V(K)|.
\end{align*}
This completes the proof of Lemma~\ref{allbigx}.
\end{proof}

The length of the signature vectors defined in Definition~\ref{sig} of
the next section (which will be our measure of the progress of
GrowTransversal) depends on the number of layers of the alternating
trees $T$ constructed by the algorithm. Our last result of this
section gives an upper bound on this quantity.

\begin{lemma}\label{fewlayers}
Suppose $(\mu, U,\rho)$ is feasible for $(r,\epsilon)$. The number of layers in the alternating tree $T$ with respect to partial independent transversal $M$ maintained during the execution of {\rm GrowTransversal} is always bounded by $c\log(m)$ where $c=\frac{1}{\log[1+\rho(1-\mu)]}$.
\end{lemma}
\begin{proof}[Proof:]
Suppose $T=(L_0,\dots,L_\ell)$ at the beginning of some iteration of
the while loop in line $7$ of GrowTransversal.
Consider any $L_i$
where $1\le i\le\ell$. By Lemma~\ref{nocollapse},
$|Y_i|>(1-\mu)|X_i|$. Since GrowTransversal did not terminate in the
iteration in which $L_i$ was constructed, by
Lemma~\ref{allbigx} we have $$|Y_i|>(1-\mu)|X_i|>\rho(1-\mu)|Y_{\le
  i-1}|.$$ 
Therefore since $|Y_1|\geq1$, we find
$m\ge|Y_{\le\ell}|=\sum\limits_{i=0}^\ell |Y_i|>
[1+\rho(1-\mu)]^\ell$. Thus the number of layers $\ell$ at any moment
of the algorithm is bounded above by $\frac{\log
  m}{\log[1+\rho(1-\mu)]}$. As $\mu$ and $\rho$ are fixed constants,
this is $c\log(m)$ for $c=\frac{1}{\log[1+\rho(1-\mu)]}$.  
\end{proof}

\section{Signatures}\label{sigs}
Recall from Section~\ref{algs} that $r$ and $\epsilon$ are fixed constants, and  $\mu$, $U$, and $\rho$ are fixed such that $(\mu, U,\rho)$ is feasible for $(r,\epsilon)$. We begin this section by defining the signature vector of an alternating tree. We then use these signature vectors to prove that GrowTransversal terminates after a polynomial in $m$ number of iterations, where the degree of the polynomial is a function of $r$ and $\epsilon$.

\begin{definition}\label{sig}
Let $T=(L_0,\dots,L_\ell)$ be an alternating tree with respect to PIT $M$ and vertex class $A$. The {\it signature} of layer $L_i$ is defined to be $$(s_{2i-1},s_{2i})=\left(-\left\lfloor \log_b \frac{\rho^{-i}}{(1-\mu)^{i-1}}|X_i|\right\rfloor, \left\lfloor \log_b \frac{\rho^{-i}}{(1-\mu)^i}|Y_i|\right\rfloor\right)$$ where $b=\frac{U}{U-\mu\rho}$.
The {\it signature vector} of $T$ is ${s=(s_1,s_2,\dots,s_{2\ell-1}, s_{2\ell},\infty)}$.
\end{definition}

The above definition for the signature of a layer $L_i$ is chosen so that the lexicographic value of the signature vector decreases whenever $|X_i|$ increases significantly (see Lemma~\ref{lexreduce} subcase 2.2) as well as decreases whenever $|Y_i|$ decreases significantly (Lemma~\ref{lexreduce} subcase 2.1). The factors $\frac{\rho^{-i}}{(1-\mu)^{i-1}}$ and $\frac{\rho^{-i}}{(1-\mu)^{i}}$ are present to ensure that the coordinates of the signature vector are non-decreasing in absolute value, which we will show in Lemma~\ref{absval}. These two properties, together with Lemma~\ref{fewlayers}, combine to give the desired upper bound on the total number of signature vectors (Lemma~\ref{fewsigs}).

We begin by showing that the lexicographic value of the signature vector decreases during each iteration of GrowTransversal.

\begin{lemma}\label{lexreduce}
The lexicographic value of the signature vector reduces across each iteration of the loop in line 7 of {\rm GrowTransversal} unless the algorithm terminates during that iteration.
\end{lemma}
\begin{proof}[Proof:]
Let $T=(L_0,\dots, L_\ell)$ and $M$ be the alternating tree and PIT at the beginning of some iteration of the while loop in line $7$ of GrowTransversal. Let $s=(s_1,\dots, s_{2\ell},\infty)$ be the signature vector of $T$. There are two cases.

\begin{case}
No collapse operation occurs in this iteration, i.e. ${|I_M(X_{\ell+1})|\le\mu|X_{\ell+1}|}$.
\end{case}
The only modification to $T$ in this iteration is a new layer $L_{\ell+1}$ is added to $T$ (lines 8-10). The new signature vector for $T$ is therefore $s'=(s'_1,\dots,s'_{2\ell+2},\infty)$ where $s'_i=s_i$ for all $1\le i\le 2\ell$ and $(s'_{2\ell+1},s'_{2\ell+2})$ is the signature of layer $L_{\ell+1}$. Hence the lexicographic value is reduced.

\begin{case}
At least one collapse operation occurs in this iteration, i.e. ${|I_M(X_{\ell+1})|>\mu|X_{\ell+1}|}$.
\end{case}
Consider the alternating tree $T^*$ returned by SuperposedBuild during the last iteration of the loop of collapses (lines 15-28) in this iteration of GrowTransversal. Either SuperposedBuild returns the same tree as its input $T'$ or SuperposedBuild modifies a layer of $T$ and removes all later layers (i.e. $T^*=(L_0,\dots,L_k)$ for some $k\le\ell$ where $L_k=(X'_k,Y'_k)$).

\subcase{2.1.}
SuperposedBuild returns its input $T'$ as $T^*$. \vspace{4pt}

Let $t$ be the index of the layer of $T$ satisfying $|I_M(X_{t})|>\mu|X_{t}|$ in the final iteration of the loop of collapses. If $t=1$, a vertex in $I_M(X_1)$ is added to $M$ and the algorithm terminates. Otherwise, $t>1$ and $T'=(L'_0,\dots,L'_{t-1})$, where $L'_i=L_i$ for all $0\le i<t-1$ and $L'_{t-1}$ is modified from $L_{t-1}$ as lines 20-23 describe. 

The signature vector of $T'$ is therefore $s'=(s'_1,\dots,s'_{2t-2},\infty)$ where $s'_i=s_i$ for all $1\le i\le 2t-4$ and $$(s'_{2t-3},s'_{2t-2})=\left(-\left\lfloor \log_b \frac{\rho^{-(t-1)}}{(1-\mu)^{t-2}}|X'_{t-1}|\right\rfloor, \left\lfloor \log_b \frac{\rho^{-(t-1)}}{(1-\mu)^{t-1}}|Y'_{t-1}|\right\rfloor\right).$$ Since lines 20-23 do not modify $X_\ell$ in GrowTransversal, $X_{t-1}$ is not modified and so $X'_{t-1}=X_{t-1}$. 

Note that each vertex class containing a vertex in $I_M(X_t)$ must contain a vertex in $Y_{t-1}$. Thus, since a vertex class contains at most $U$ vertices in $X_t$, there are at least $\frac{\mu}{U}|X_t|$ blocking vertices in $Y_{t-1}$ that are not in $T'$ because of the replacements in lines 20-23. By Lemma~\ref{allbigx}, we have $|X_t|>\rho|Y_{\le t-1}|\ge\rho|Y_{t-1}|$ and so $$|Y'_{t-1}|\le|Y_{t-1}|-\frac{\mu}{U}|X_t|<\left(1-\frac{\mu\rho}{U}\right)|Y_{t-1}|.$$ Therefore, 
\begin{align*}
\log_b \frac{\rho^{-(t-1)}}{(1-\mu)^{t-1}}|Y'_{t-1}|&< \log_b \frac{\rho^{-(t-1)}}{(1-\mu)^{t-1}}\left(1-\frac{\mu\rho}{U}\right)|Y_{t-1}| \\
&\le \log_b \left(1-\frac{\mu\rho}{U}\right) + \log_b \frac{\rho^{-(t-1)}}{(1-\mu)^{t-1}}|Y_{t-1}|\\
&\le \log_b \left(\frac{U-\mu\rho}{U}\right) + \log_b \frac{\rho^{-(t-1)}}{(1-\mu)^{t-1}}|Y_{t-1}|\\
&=-1+\log_b \frac{\rho^{-(t-1)}}{(1-\mu)^{t-1}}|Y_{t-1}|.
\end{align*}
Hence $s'_{2t-3}=s_{2t-3}$ and $s'_{2t-2}<s_{2t-2}$, so the lexicographic value of the signature vector is reduced.

\subcase{2.2.}
SuperposedBuild returns an alternating tree $T^*$ different from its input $T'$. \vspace{4pt}

Let $t$ be the index of the layer of $T$ satisfying $|I_M(X_{t})|>\mu|X_{t}|$ in the final iteration of the loop of collapses. Let $q$ be the index of the last layer in the alternating tree returned by SuperposedBuild in the final iteration of the loop of collapses. Hence $|X'_{q}|\ge (1+\mu)|X_{q}|$. 

As $(\mu, U,\rho)$ is feasible, $\rho\le U-\mu\rho$. Hence $\frac{U}{U-\mu\rho}=1+\frac{\mu\rho}{U-\mu\rho}\le 1+\mu$. Thus,
\begin{align*}
\log_b \frac{\rho^{-q}}{(1-\mu)^{q-1}}|X'_q|&\ge \log_b \frac{\rho^{-q}}{(1-\mu)^{q-1}}(1+\mu)|X_{q}| \\
&=\log_b\left(1+\mu\right) + \log_b \frac{\rho^{-q}}{(1-\mu)^{q-1}}|X_{q}|\\
&\ge1+\log_b \frac{\rho^{-q}}{(1-\mu)^{q-1}}|X_{q}|
\end{align*}
and so $s'_{2q-1}<s_{2q-1}$. Since SuperposedBuild and the loop of collapses do not modify layers $L_0,\dots, L_{q-1}$, we see that $s'_i=s_i$ for all $1\le i\le 2q-2$. Hence the signature vector of $T^*$ returned by SuperposedBuild is $(s'_1,\dots,s'_{2q-1},s'_{2q},\infty)$. Thus the lexicographic value of the signature vector is reduced.
\end{proof}

\begin{lemma}\label{absval}
The coordinates of the signature vector are non-decreasing in absolute value at the beginning of each iteration of the while loop in line $7$ of {\rm GrowTransversal}.
\end{lemma}
\begin{proof}[Proof:]
Let $T=(L_0,\dots, L_\ell)$ and $M$ be the alternating tree and PIT at the beginning of some iteration of the while loop in line $7$ of GrowTransversal. Consider layer $L_i$ for some $1\le i\le\ell$. Since $|Y_i|\ge(1-\mu)|X_i|$ by Lemma~\ref{nocollapse}, $$|s_{2i-1}|=\left\lfloor \log_b \frac{\rho^{-i}}{(1-\mu)^{i-1}}|X_i|\right\rfloor\le\left\lfloor \log_b \frac{\rho^{-i}}{(1-\mu)^{i}}|Y_i|\right\rfloor=|s_{2i}|.$$ Hence the coordinates of the signature vector for a layer of $T$ are non-decreasing in absolute value. Now consider layers $L_i$ and $L_{i+1}$ for some $0\le i\le\ell-1$. Lemma~\ref{allbigx} implies $|X_i|\ge\rho|Y_{i-1}|$ and so $$|s_{2i}|=\left\lfloor \log_b \frac{\rho^{-i}}{(1-\mu)^{i}}|Y_i|\right\rfloor\le\left\lfloor \log_b \frac{\rho^{-(i+1)}}{(1-\mu)^{i}}|X_{i+1}|\right\rfloor=|s_{2i+1}|.$$ Thus consecutive coordinates of the signature vector for coordinates of different layers of $T$ are also non-decreasing in absolute value. Hence the coordinates of the signature vector of $T$ are non-decreasing in absolute value.
\end{proof}

We may now use Lemmas~\ref{lexreduce} and~\ref{absval} to bound the
total number of possible signature vectors. 

\begin{lemma}\label{fewsigs}
Let $T=(L_0,\dots, L_\ell)$ and $M$ be the alternating tree and
partial independent transversal at the beginning of some iteration of
the while loop in line $7$ of GrowTransversal. The number of possible
signature vectors for $T$ is bounded by a polynomial in $m$ of degree
$k$ where $k$ depends only on $r$ and $\epsilon$. 
\end{lemma}
\begin{proof}[Proof:]
For each layer of $T$, the signature vector of $T$ contains two
coordinates. Thus, by Lemma~\ref{fewlayers}, the signature vector of
$T$ has at most $2c\log m$ coordinates where
$c=\frac{1}{\log[1+\rho(1-\mu)]}$.  
Also, by Lemma~\ref{absval}, the coordinates are non-decreasing in
absolute value and so the absolute   
 value of the final (finite) coordinate is an upper bound on the absolute value
 of each coordinate in the signature vector. By Definition~\ref{sig},
 the final coordinate is $\left\lfloor
 \log_b\left(\left[\frac{\rho^{-\ell}}{(1-\mu)^\ell}\right]|Y_\ell|\right)\right\rfloor$. As $\ell\le c\log (m)$ (by Lemma~\ref{fewlayers})
 and $|Y_\ell|\le m$, the absolute value of each coordinate of the
 signature vector is bounded above by  
\begin{align*}
\log_b \left[\frac{\rho^{-c\log(m)}}{(1-\mu)^{c\log(m)}}m\right]&=\log_b m+\log_b \rho^{-c\log(m)}-\log_b(1-\mu)^{c\log(m)}\\
&=\log_b m-c[\log_b(\rho)]\log(m)-c[\log_b(1-\mu)]\log(m)\\
&=\left[\frac1{\log(b)}-c[\log_b(\rho)]-c[\log_b(1-\mu)]\right]\log(m).
\end{align*} 
Let $R=\left[\frac1{\log(b)}-c[\log_b(\rho)]-c[\log_b(1-\mu)]\right]$ and note
  that $R$ is a fixed constant that depends only on $r$ and $\epsilon$
  (since $b$ and $c$ depend only on $\mu$, $U$, and $\rho$, which in
  turn depend only on $r$ and $\epsilon$).

Now to each signature vector ${s=(s_1,s_2,\dots,s_{2\ell-1},
s_{2\ell},\infty)}$ we associate the vector 
$s^+=(s_1-1,s_2+2,\dots,s_{2\ell-1}-(2\ell-1),
s_{2\ell}+2\ell,\infty)$. Then the final coordinate of $s^+$ is at most
$R\log m+2\ell\leq(R+2c)\log m$. Since the coordinates of $s$ are
non-decreasing in 
absolute value (and considering the sign pattern), the coordinates of $s^+$
are strictly increasing in absolute value. Thus each vector $s^+$
corresponds to a distinct subset of the set
${\{1,\ldots,\lfloor(R+2c)\log m\rfloor\}}$. Hence the total number of
vectors $s^+$ (and hence the total number of signature vectors) is at
most $2^{(R+2c)\log m}$. This completes the proof.
\end{proof}

(We remark that the idea of the last paragraph was suggested by a
referee of~\cite{Annamalai2017}, see~\cite{AnnamalaiThesis, Annamalai2017}.)

We now complete the proof of Theorem~\ref{main}.

\begin{proof}[Proof of Theorem~\ref{main}:]

Let FindITorBD be the following algorithm.

\begin{algorithmic}[1]
\Function{FindITorBD}{$G;V_1,\dots,V_m$}
\State $M:=\emptyset$
\For{$i=1,\dots,m$}
	\State Choose an $A\in\{V_1,\dots, V_m\}$ such that $A\cap M=\emptyset$
	\State $(M,T,x):=\text{GrowTransversal}(M,A)$
	\If{$x=1$}
		\State ${\mathcal B}:=A(Y_{\leq\ell})$
		\For{$j=1,\dots,\ell$}
			\State $(X'_j,Y'_j):=\text{BuildLayer}((L_0,\dots,L_{j-1}),X_j,Y_j)$
			\EndFor
		\State ${\mathcal B}:={\mathcal B}\setminus A^U$
		\State ${\mathcal B}:={\mathcal B}\setminus\left(\bigcup\limits_{i=1}^\ell A(X'_i\setminus X_i)\right)$
		\State $D:=X'_{\le\ell}\cup Y'_{\le\ell}\cup X_{\ell+1}\cup Y_{\ell+1}\cup S$ as in Lemma~\ref{allbigx}
		\State \Return $\cB$, $D$ and \lq\lq $G_{\mathcal B}$ is dominated by $D$\rq\rq and terminate
	\EndIf
	\State $M:= M$ returned by GrowTransversal
\EndFor
\State \Return $M$ and \lq\lq $M$ is an independent transversal of $G$\rq\rq
\EndFunction
\end{algorithmic}\vspace{6pt}

By Lemma~\ref{lexreduce}, every iteration of GrowTransversal reduces
the lexicographic value of the signature vector of an alternating tree
$T$ with respect to a PIT $M$. Furthermore, Lemma~\ref{fewsigs}
implies that the number of such signature vectors is bounded by $m^f$
where $f$ depends only on $r$ and $\epsilon$. Thus, GrowTransversal
terminates after at most $m^f$ iterations. It can be
easily verified (see~\cite{thesis}) that 
each iteration of GrowTransversal can be implemented to run in time
$O(|V(G)|^4)$. Since  GrowTransversal is implemented at most $m$ times
in FindITorBD, and the other steps of FindITorBD are easily
implemented in $O(|V(G)|^4)$ operations, the runtime of FindITorBD is
$O(|V(G)|^4m^{f+1})$ and thus is polynomial in $|V(G)|$ because $r$
and $\epsilon$ are fixed constants.

It remains to show that FindITorBD returns one of the two stated outcomes.
FindITorBD starts with $M=\emptyset$ and runs GrowTransversal at most $m$ times. Note that because of the augmentation in line 17 of GrowTransversal, the PIT $M$ at the end of an iteration of GrowTransversal covers one more vertex class than the PIT covered at the start of the iteration. Also, the PIT at the end of one iteration of GrowTransversal is the initial PIT of the next iteration. 

Suppose FindITorBD terminates during iteration $i$ of
GrowTransversal, so that GrowTransversal terminates before completing the
$m^\text{th}$ iteration. Then the $i^\text{th}$ iteration of
GrowTransversal returned $(M,T,1)$ for some alternating tree $T$ and
PIT $M$. Since the sets $\mathcal B$ and $D$ are defined
identically to how they are defined in the proof of
Lemma~\ref{allbigx}, we then have that $D$ dominates $G_{\mathcal B}$,
and (as stated in Lemma~\ref{allbigx}) $\cB$ and $D$ have the
properties stated in the conclusions of
Theorem~\ref{main}. 

Suppose FindITorBD does not terminate before completing all $m$
iterations of GrowTransversal. Then all $m$ vertex classes contain a
vertex in the final \lq\lq partial\rq\rq\hspace{1pt} IT $M$. Hence $M$
is an IT of $G$. Therefore, FindITorBD returns an IT in $G$. 
\end{proof}

\section{Applications}\label{apps}
In this section, we briefly discuss some applications of
Theorem~\ref{main}. In particular, we discuss applications to
hypergraph matching (Subsection~\ref{matchings}),  
circular chromatic index (Subsection~\ref{circular}), strong colouring
(Subsection~\ref{strongcolour}), 
and hitting
sets for maximum cliques (Subsection~\ref{maxcliques}). Precise
  details of these applications, as well as others, appear
  in~\cite{thesis}.

\subsection{Hypergraph Matchings}\label{matchings}

Here we consider a hypergraph version of Hall's Theorem for bipartite graphs.
An {\it $r$-uniform bipartite hypergraph} $H=(A,B,E)$ is a hypergraph
on a vertex set that is partitioned into two sets $A$ and $B$ such
that ${|e\cap A|=1}$ and ${|e\cap B|=r-1}$ for each edge
$e\in E$. A {\it perfect matching} in $H$ is a subset $M\subseteq E$
of pairwise disjoint edges of $H$ that saturates $A$, in other words
$|M|=|A|$. For a set
$S\subseteq A$, we write $E_S=\{e\in E:|e\cap S|=1\}$ for the set of
hyperedges in $H$ incident to $S$. For a collection of edges
$F\subseteq E$, we denote by  $\tau_B(F)$ the smallest cardinality of a
$B$-{\it cover} of $F$, that is, a subset
$T\subseteq B$ such that $|e\cap T|\neq\emptyset$ for each $e\in F$.
The following generalisation of Hall's Theorem
from~\cite{Haxell1995} 
provides a condition under which $H$ admits a perfect matching. 

\begin{theorem}\label{hyper}
Let $H=(A,B,E)$ be an $r$-uniform bipartite hypergraph. If $$\tau_B(E_S)>(2r-3)(|S|-1)$$ for all $S\subseteq A$, then $H$ admits a perfect matching.
\end{theorem}
It was shown in~\cite{Haxell1995} that Theorem~\ref{hyper} is best
possible for every $r$. Note that when $r=2$ it is (the nontrivial
direction of) Hall's Theorem. 

In fact Theorem~\ref{hyper} is a special case of
Theorem~\ref{maxtrans}, as can be seen from the following
argument. Given $H=(A,B,E)$, construct an auxiliary graph $G^H$ with
vertex set $E$, in which vertices $e$ and $f$ are adjacent if and only
if $e\cap f\cap B\not=\emptyset$. Consider the vertex partition
of $G^H$ given by assigning $e$ and $f$ to the same
vertex class if and only if $e\cap f\cap A\not=\emptyset$. Thus
the vertex classes are indexed by $A$. With these definitions, a set
$M\subseteq E$ is a perfect 
matching of $H$ if and only if $M\subseteq V(G^H)$ is an IT of $G^H$.

By Theorem~\ref{maxtrans} applied to $G^H$ (using also the comment before
Definition~\ref{constell}), 
if $H$ does not have a perfect matching, then there exists a subset
$\cB$ of vertex classes (indexed by a set $S(\cB)\subseteq A$) that is
dominated by $V(K)$ for a 
constellation $K$ of $\cB$. Thus $T=\bigcup_{e\in V(K)}e\cap B$ is a set of
vertices of $H$ forming a $B$-cover of $E_{S(\cB)}$. Then the
following claim gives an immediate contradiction to the assumption of
Theorem~\ref{hyper}, thus completing the proof.
\begin{claim}\label{cover}
Let $\cB$ be a subset of the vertex classes of $G^H$, and let $K$ be
constellation for $\cB$. Then 
$|\bigcup_{e\in V(K)}e\cap B|\leq(2r-3)(|{S(\cB)}|-1)$.
\end{claim}
\begin{proof}[Proof:]
 Each component $C$ of $K$ corresponds to a
set of edges of $H$, consisting of the centre $e_C$ of the star $C$
and a nonempty set 
$L_C$ of leaves, all of which intersect $e_C$ in $B$. Hence the total
number of vertices of $B$ contained in $\{e_C\}\cup L_C$ is at most 
$(r-1)+(r-2)|L_C|$. By definition $\bigcup_CL_C$ is an IT of $|\cB|-1$
classes of $\cB$, implying that $K$ has 
at most $|{S(\cB)}|-1$ components, and that
$\sum_C|L_C|=|{S(\cB)}|-1$. Therefore 
\begin{align*}
\left|\bigcup_{e\in V(K)}e\cap B\right|&\leq\sum_{C}\left(r-1+(r-2)|L_C|\right)\\
   &\leq(r-1)(|{S(\cB)}|-1)+(r-2)(|{S(\cB)}|-1)=(2r-3)(|{S(\cB)}|-1),
\end{align*}
where the sum is over all components $C$ of $K$.
\end{proof}

The following algorithmic version of Theorem~\ref{hyper} was proved
by Annamalai in~\cite{Annamalai2016, Annamalai2017}.
\begin{theorem}\label{r-partite}
For every fixed choice of $r\ge2$ and $\epsilon>0$, there exists an
algorithm $\mathcal A$ that finds, in time polynomial in the size of
the input, a perfect matching in $r$-uniform bipartite hypergraphs
$H=(A,B,E)$ satisfying $$\tau_B(E_S)>(2r-3+\epsilon)(|S|-1)$$ for all
$S\subseteq A$. 
\end{theorem}

Here we show that Theorem~\ref{main} is a generalisation of
Theorem~\ref{r-partite}. First note that for every
$r$-uniform bipartite hypergraph $H=(A,B,E)$, the graph $G^H$
is $r$-claw-free with respect to any partition. Indeed,
the neighbours of $e$ forming any independent set in $G^H$
must all contain
distinct vertices of $e\cap B$, and $|e\cap B|=r-1$. Thus  given $r$
and $\epsilon$, we may apply Theorem~\ref{main} with $r$ and
$\epsilon'=\epsilon/(r-1)$ to obtain a polynomial-time algorithm $\mathcal
A$ that finds for each input $H=(A,B,E)$ either
\begin{enumerate}
\item an IT in $G^H$, which is a perfect matching
  in $H$, or
\item a set $\mathcal B$ of vertex classes and a set $D$ of vertices
  of $G^H$ such that $D$ dominates $G_{\cB}$ in $G^H$ and
  $|D|<(2+\epsilon')(|{\mathcal B}|-1)$. Moreover $D$ contains $V(K)$
  for a constellation $K$ of some $\cB_0\supseteq\cB$, where
  $|D\setminus V(K)|<\epsilon'(|\cB|-1)$.
\end{enumerate}
If (1) is the outcome for every input $H$ then $\mathcal A$ is the required
algorithm, so suppose (2) holds for some $H$. Then $D$ is a set of
edges of $H$ such that every edge of $E_{S(\cB)}$ intersects
$T=\bigcup_{e\in D}e\cap B$. For $u=|D\setminus V(K)|$ it is clear that
$|\bigcup_{e\in D\setminus V(K)}e\cap B|\leq u(r-1)$.

Next we estimate
$|\bigcup_{e\in V(K)}e\cap B|$.
As in the proof of Claim~\ref{cover}, for
each component $C$ of the constellation $K$, the number of vertices of
$B$ contained in $\{e_C\}\cup L_C$ is at most
$(r-1)+(r-2)|L_C|=1+(r-2)|V(C)|$. Each component of $K$ has at least
two vertices, so the number of components is at most $|V(K)|/2=(|D|-u)/2$.
Since $|D|<(2+\epsilon')(|{\mathcal B}|-1)$ we find 
\begin{align*}
\left|\bigcup_{e\in V(K)}e\cap B\right|&\leq\sum_C\left(1+(r-2)|V(C)|\right)\\
         &\leq|V(K)|/2+(r-2)|V(K)|\\
      &=(2r-3)|V(K)|/2\\
      &=(2r-3)(|D|-u)/2<(2r-3)(1+\epsilon'/2)(|{\mathcal B}|-1)-u(r-3/2),
\end{align*}
where again the sum is over all components $C$ of $K$. Therefore
\begin{align*}|T|&<u(r-1)+(2r-3)(1+\epsilon'/2)(|{\mathcal
    B}|-1)-u(r-3/2)\\
&= u/2+(2r-3+\epsilon'(r-3/2))(|\cB|-1)\\
   &<(2r-3+\epsilon'(r-1))(|\cB|-1)= (2r-3+\epsilon)(|\cB|-1),
\end{align*}
where in the last line we used that $u<\epsilon'(|\cB|-1)$ and
$\epsilon'=\epsilon/(r-1)$.  
But then this contradicts the assumption
$\tau_B(E_S)>(2r-3+\epsilon)(|S|-1)$ for $S=S(\cB)$, thus proving
Theorem~\ref{r-partite}.

\subsection{Circular Chromatic Index}\label{circular}
A {\it proper circular $p/q$-edge-colouring} of a graph $G$ is a colouring of the edges of $G$ with colours in $\{0,\dots,p-1\}$ such that the difference modulo $p$ of the colours assigned to two adjacent edges is not in $\{-(q-1),-(q-2),\dots,q-1\}$. The smallest ratio $p/q$ for which there is a proper circular $p/q$-edge-colouring of $G$ is called the {\it circular chromatic index} of $G$.

In \cite{Kaiser2004}, Kaiser, Kr\'{a}l, and \v Skrekovski proved the
following result using Theorem~\ref{maxtrans}. 
\begin{theorem}\label{circcubic}
Let $p\in\mathbb N$ with $p\ge2$ and $G$ be a cubic bridgeless graph with girth $$g=\begin{cases}2(2p)^{2p-2}&\text{if }p\ge2\text{ is even}\\2(2p)^{2p}&\text{if }p\ge3\text{ is odd.} \end{cases}$$ Then $G$ admits a proper circular $(3p+1)/p$-edge-colouring.
\end{theorem}


To prove Theorem~\ref{circcubic}, Kaiser, Kr\'{a}l, and \v Skrekovski
proved the existence of an IT in a certain auxiliary graph $G'$ constructed
using $G$, $p$, and a fixed 1-factor $F$ of $G$. The vertices of $G'$
are partitioned into $m$ classes $V_i$, one for each odd cycle $C_i$ of
$G-F$, where $|V_i|=|V(C_i)|$. The maximum degree of $G'$ is at most
$(2p)^{2(p-1)}$ if $p$ is even, and at most $(2p)^{2p}$ if $p$ is
odd. Hence by Theorem~\ref{maxdeg} there exists an IT
in $G'$ provided $|V_i|\geq 2(2p)^{2(p-1)}$ (respectively $|V_i|\geq
2(2p)^{2p}$) for each $i$. Given such an 
IT in $G'$, the authors explicitly provide the required proper circular
$(3p+1)/p$-edge-colouring of $G$.

We may apply Corollary~\ref{maxdegalg} (the algorithmic version of
Theorem~\ref{maxdeg}) to obtain an algorithmic
version of Theorem~\ref{circcubic}, and in fact in this case no
weakening of the result at all is necessary. This is because the
sizes of the vertex classes $V_i$ are exactly the lengths of the odd
cycles in $G-F$, and hence girth $g\geq 2(2p)^{2(p-1)}$ (respectively
$g\geq 2(2p)^{2p}$) is enough to ensure the extra one in the lower
bounds on the $|V_i|$ required by Corollary~\ref{maxdegalg}. Note also that the number $m$ of odd cycles in $G-F$ is clearly less than $|V(G)|$. 

\begin{corollary}\label{largerg}
Let $p\in\mathbb N$ with $p\ge2$ be given. Then there exists an algorithm that takes as input any cubic bridgeless graph $G$
with girth $$g=\begin{cases}2(2p)^{2p-2}&\text{if }p\ge2\text{ is
  even}\\2(2p)^{2p}&\text{if }p\ge3\text{ is odd} \end{cases}$$ 
and finds, in time polynomial in $|V(G)|$, a proper circular
$(3p+1)/p$-edge-colouring of $G$.
\end{corollary}

\subsection{Strong Colouring}\label{strongcolour}
Let $k$ and $n$ be positive integers. Let $G$ be a graph with $n$
vertices and let $(V_1,\dots,V_m)$ be a vertex partition of $V(G)$ such
that ${|V_i|\le k}$ for all $i$. A graph $G$ is {\it strongly
  $k$-colourable with respect to $(V_1,\dots,V_m)$}
if there is a colouring of $G$ with $k$ colours so that for each vertex
class, each colour is assigned to at most one vertex of the class. If
$G$ is strongly $k$-colourable with respect to every vertex partition
of $V(G)$ into classes of size at most $k$, we say $G$ is {\it
  strongly $k$-colourable}. The {\it strong chromatic number} of a
graph $G$, denoted $s\chi(G)$, is the minimum $k$ such that $G$ is
strongly $k$-colourable. This notion was introduced independently by
Alon~\cite{Alon1988, Alon1992} and Fellows~\cite{Fellows1990}.

The best known general bound for the strong chromatic number of graphs
$G$ in
terms of their maximum degree $\Delta(G)$ is
$s\chi(G)\leq3\Delta(G)-1$, proved in~\cite{Haxell2004}. (See
also~\cite{Haxell2008} for an asymptotically better bound.)
In \cite{Aharoni2007}, Aharoni, Berger, and Ziv gave a nice simplification
of the proof in~\cite{Haxell2004}, that gives the bound
$s\chi(G)\leq3\Delta(G)$. 
Their argument uses a slight strengthening of Theorem~\ref{maxdeg},
which states that if $G$ is a graph with maximum degree $\Delta$, and
if  $|V_i|\ge2\Delta$ for each $i$, then for each vertex $v$ there exists
an IT of $G$ containing $v$. This follows immediately from
Theorem~\ref{maxtrans} applied to $G$ with the partition
$(\{v\},V_2,\ldots,V_m)$, assuming without loss of generality that
$v\in V_1$. (See the note after the statement of
Theorem~\ref{maxtrans} in the Introduction.) 

To make the argument of~\cite{Aharoni2007} algorithmic, we just apply
Theorem~\ref{main} instead of Theorem~\ref{maxtrans} in the previous
paragraph. This gives the following slightly strengthened version of
Corollary~\ref{maxdegalg}. 

\begin{corollary}\label{hasv}
Let $\Delta\in\mathbb N$ be given. Then there exists an algorithm $\mathcal A$
that takes as input any graph $G$ with maximum degree $\Delta$ and vertex
 partition $(V_1,\ldots,V_m)$ such that ${|V_i|\ge2\Delta+1}$ for each
 $i$, and any $v\in V(G)$,
 and finds, in time polynomial in $|V(G)|$, an independent transversal in
 $G$ that contains $v$.
\end{corollary}

The proof in~\cite{Aharoni2007} begins with a partial strong
$3\Delta$-colouring $c$ of a graph $G$ with respect to a vertex partition
$(V_1,\dots,V_m)$, an uncoloured vertex $v$, and a colour $\alpha$ not
used by $c$ on the vertex class of $v$. A new
graph $G'$ is obtained by removing from each $V_i$ the vertices whose
colour appears on the neighbourhood of the vertex $w_i$ in $V_i$ coloured
$\alpha$ (if it exists). This reduces the size of each class by at most $\Delta$. Then the strengthened version of Theorem~\ref{maxdeg} is used
to find an IT $Y$ of $G'$ containing $v$. As shown
in~\cite{Aharoni2007}, the modification of $c$ obtained by giving 
each $y_i\in Y\cap V_i$ colour $\alpha$, and each $w_i$ colour
$c(y_i)$, is a partial strong $3\Delta$-colouring that colours more
vertices than $c$ did (in particular it colours $v$). Hence in at most
$|V(G)|$ such steps a suitable 
colouring is constructed. This argument therefore gives the following
corollary of Corollary~\ref{hasv}.

\begin{corollary}\label{strongalg}
Let $\Delta$ be a positive integer. There exists an algorithm
$\cA$ that takes as input any graph $G$ with maximum degree $\Delta$
and vertex partition $(V_1,\dots,V_m)$ where $|V_i|\le3\Delta+1$ for
each $i$, and finds, in time polynomial
in $|V(G)|$, a strong $(3\Delta+1)$-colouring of $G$ with respect to
$(V_1,\dots,V_m)$. 
\end{corollary}

\subsection{Hitting Sets for Maximum Cliques}\label{maxcliques}
It was first shown by
Rabern~\cite{Rabern2011}, and later (with a best possible bound) by
King~\cite{King2011}, that when the maximum degree and clique number of
a graph are close enough, the graph contains an independent set
meeting all maximum cliques. Finding such a set is important for various colouring problems (see e.g.~\cite{King2011, Reed2001}). King's result is as follows,
where $\omega(G)$ denotes the clique number of the graph $G$. 

\begin{theorem}\label{clique}
Let $G$ be a graph of maximum degree $\Delta$ such that
${\omega(G)>\frac{2}{3}(\Delta+1)}$. Then $G$ contains an independent
set meeting every maximum clique. 
\end{theorem}

As shown in~\cite{King2011}, if ${\omega(G)>\frac{2}{3}(\Delta+1)}$
then the set of all maximum cliques in $G$ can be partitioned into
classes such that for each class $\cC_i$ we have
$|\cap_{C\in\cC_i}V(C)|\geq(\Delta+1)/3$. Thus for each class $\cC_i$
there exists a  
``core'' $V_i$ of at least $(\Delta+1)/3$ vertices that is contained
in every clique in $\cC_i$. Moreover the $V_i$ are all disjoint. Thus
an IT of the subgraph of $G$ induced by the union of
all the cores $V_i$ (where the $V_i$ are the vertex classes) provides
a suitable independent hitting set. By following the argument
in~\cite{King2011} (and being slightly more careful with divisibility),
it can be shown that it suffices to establish the following
modification of Theorem~\ref{maxtrans} with the value
$k=\lceil(\Delta+1)/3\rceil$. 

\begin{lemma}\label{boundeddeg} 
Let $k$ be a positive integer and let $G$ be a graph with vertex partition
$(V_1,\dots,V_m)$. If for every $i$ and every $v\in V_i$, the vertex $v$ has at
most $\min\{k-1,|V_i|-k\}$ neighbours outside $V_i$, then $G$ has an IT.
\end{lemma}
In~\cite{King2011} the slightly stronger statement with
$\min\{k-1,|V_i|-k\}$ replaced by $\min\{k,|V_i|-k\}$ is established
and used, but the proof is not algorithmic. The advantage of
Lemma~\ref{boundeddeg} is that we can use Theorem~\ref{main} to give
an algorithmic proof, as follows. 

\begin{proof}[Proof:] As before we may assume each $V_i$ is independent.
Since each vertex has at most $k-1$ neighbours, we have that $G$ is
$k$-claw-free. Let $\epsilon=\frac{1}{k-1}$. We apply the
algorithm FindITorBD to $G$ and $(V_1,\dots,V_m)$. Since $k$ is fixed,
the running
time is polynomial in $|V(G)|$, where the degree of the polynomial
depends only on $k$. We obtain either
\begin{enumerate}
\item an IT of $G$, or
\item a set $\mathcal B$ of vertex classes and a set $D$ of vertices
  of $G$ such that $D$ dominates $G_{\cB}$ in $G$ and
  $|D|<(2+\epsilon)(|{\mathcal B}|-1)$. Moreover $D$ contains $V(K)$
  for a constellation $K$ of some $\cB_0\supseteq\cB$, where
  $|D\setminus V(K)|<\epsilon(|\cB|-1)$.
\end{enumerate}
If (1) is the output for every input $G$ then we have the required
algorithm, so suppose (2) holds for some $G$. Recall from
Definition~\ref{constell} that the set of leaves in the constellation $K$
forms an IT of $|\cB_0|-1$ vertex classes of $\cB_0$, and hence in
particular $D$ contains an IT $Y$ of a set $\cB'$ of $|\cB|-1$ vertex
classes of $\cB$.  

Since $D$
dominates $G_{\mathcal B}$,
we know $$\sum\limits_{v\in D}d(v)\ge\sum\limits_{V_i\in {\mathcal
    B}}|V_i|\ge\sum\limits_{V_i\in {\mathcal B}'}|V_i|.$$  

Since $Y$ is an IT of ${\mathcal B}'$, we have that $|Y|\ge|{\mathcal
  B}|-1$ and so $|D\setminus Y|<(1+\epsilon)(|{\mathcal B}|-1)$. For
any vertex $u\in D\setminus Y$, we know $d(u)\le k-1$ and for any
vertex $v\in Y$, we have $d(v)\le |V_i|-k$ where $v\in V_i$. Hence, 
\begin{align*}
\sum\limits_{v\in D}d(v)&\le\sum\limits_{v\in Y}d(v)+(k-1)|D\setminus Y|\\
&<\sum\limits_{V_i\in {\mathcal B}'}\left(|V_i|-k\right)+(k-1)(1+\epsilon)(|{\mathcal B}|-1)\\
&\le\sum\limits_{V_i\in {\mathcal B}'}|V_i|-(|\cB|-1)+\epsilon (k-1)(|{\mathcal B}|-1)\\
&=\sum\limits_{V_i\in {\mathcal B}'}|V_i|.
\end{align*}
This is a contradiction and so outcome (2) never occurs. This
completes the proof.
\end{proof}

When $\Delta$ is fixed, the sets $V_i$ can be found algorithmically in
time polynomial in $|V(G)|$. Thus we have the following.

\begin{theorem}\label{cliquealg}

\begin{enumerate}
\item Let $k$ be a positive integer. There exists an algorithm $\cA$ that
finds an IT in any graph $G$ with vertex partition
$(V_1,\dots,V_m)$ with the property that, for each $i$ and each $v\in
V_i$, the vertex $v$ has at
most $\min\{k-1,|V_i|-k\}$ neighbours outside $V_i$. The running time
of $\cA$ is polynomial in $|V(G)|$.
\item Let $\Delta$ be a positive integer. There exists an algorithm
  $\cA'$ that
finds, in time polynomial in $|V(G)|$, an
independent set meeting every maximum clique in any graph $G$ with
maximum degree $\Delta$ and $\omega(G)>2(\Delta+1)/3$.
\end{enumerate}
\end{theorem}

\section{Concluding remarks}\label{concrem}

As noted in the introduction, the running time of the algorithm in
Theorem~\ref{main} is polynomial in $|V(G)|$,
where the degree of the polynomial depends on the input parameters
$\epsilon$ and $r$. (Similarly in Corollary~\ref{maxdegalg} it depends
on $\Delta$). While the dependence on $\epsilon$ seems
unavoidable, we are not certain of the nature (or even the
necessity) of the dependence on $r$. If this could be avoided, or even
if the condition of being $r$-claw-free could be substantially
weakened, then algorithmic versions of many more applications of
Theorems~\ref{maxdeg} and~\ref{maxtrans} would follow. The
$r$-claw-free condition
was required by our choice of signature vector, and used in the bounds
in Lemma~\ref{allbigx}. We consider it
an interesting open question as to whether this condition is essential.
Recently, a randomised algorithm (that uses FindITorBD as a subroutine) was developed in~\cite{GHH} that overcomes the exponential dependence on $\Delta$ in Corollary~\ref{maxdegalg}.

See~\cite{thesis} for full details on the applications outlined in
Section~\ref{apps} as 
well as others (such as e.g. special graph partitions as
in~\cite{Alon+2003}, and more specific results on circular chromatic
index as in~\cite{Kaiser2004}), and more discussion of other applications of 
Theorems~\ref{maxdeg} and~\ref{maxtrans}. 

We close by pointing out that Theorems~\ref{maxdeg} and~\ref{maxtrans}
have both a {\it combinatorial} proof~\cite{Haxell1995} and a {\it
  topological} proof~\cite{Aharoni2002, Haxell2001, Meshulam2001}, 
using the notion of {\it topological connectedness} (see
e.g.~\cite{Aharoni2006}). The algorithms 
presented in this paper, and also that of~\cite{Annamalai2016,
AnnamalaiThesis, Annamalai2017}, are based on the combinatorial proofs of these
results. There are other criteria guaranteeing the existence of
independent transversals for which only a topological proof is known,
for example~\cite{Aharoni2000, Aharoni2006}, which also
have many applications. Thus algorithmic versions of these results
would also be very interesting and useful, but currently seem out of
reach. One simply stated example concerns 3-{\it partite} 3-uniform
hypergraphs, in which the vertex set has a partition into 3 parts $A$,
$B$, and $C$, and each edge contains exactly one vertex from each
part. (Thus a 3-partite 3-uniform hypergraph is also bipartite in the
sense of Section~\ref{matchings}.) The following is a slight
reformulation of the 
main result of~\cite{Aharoni2001}, which is a direct application
of~\cite{Aharoni2000}. 

\begin{theorem}
Let $H$ be a 3-partite 3-uniform hypergraph and let $k$ be a non-negative
integer. Then $H$ contains either a set of $k$ disjoint edges, or a
set $W$ of vertices such that $|W|\leq2(k-1)$ and $W\cap e\not=\emptyset$
for every edge $e$ of $H$. 
\end{theorem}
If the requirement $|W|\leq 2(k-1)$ is relaxed to $|W|\leq3(k-1)$ then a
simple greedy matching procedure gives an algorithmic proof. The
best known algorithmic result for this problem is given by the proof in~\cite{HaxellPMH}, which is also combinatorial, and provides a suitable algorithm
for $|W|\leq 5(k-1)/2$.  

\bigskip

\noindent{\bf{Acknowledgements:}} The authors are indebted to
Chidambaram Annamalai for very helpful discussions and
correspondence. They would also like to thank Nikhil Bansal and David Harris for helpful comments.

\end{document}